\documentclass[a4paper,reqno,12pt]{amsart}
\usepackage{fancyhdr}
\usepackage{amsmath}
\usepackage{dsfont}
\usepackage{hyperref}
\usepackage[english]{babel}
\usepackage{cite,enumerate,float,indentfirst}
\usepackage{graphicx}
\usepackage{latexsym,a4,mathrsfs,amsthm,amsmath,amssymb,url}
\usepackage{amsfonts}
\usepackage{amssymb}
\usepackage{epstopdf}

\newcommand\NoBlackBoxes{\global\overfullrule0pt}
\NoBlackBoxes
\numberwithin{equation}{section}

\newtheorem{theorem}{Theorem}[section]
\newtheorem{proposition}[theorem]{Proposition}
\newtheorem{lemma}[theorem]{Lemma}
\newtheorem{definition}[theorem]{Definition}

\newtheorem{condition}[theorem]{Conditions}
\newtheorem{remark}[theorem]{Remark}
\newtheorem{example}[theorem]{Example}

\begin{document}
	
\vspace{2mm}

\title{Second Order Concentration via Logarithmic Sobolev Inequalities}
\author{F. G\"{o}tze}
\address{Friedrich G\"{o}tze, Faculty of Mathematics, Bielefeld University, Bielefeld, Germany}
\email{goetze@math.uni-bielefeld.de}
\author{H. Sambale}
\address{Holger Sambale, Faculty of Mathematics, Bielefeld University, Bielefeld, Germany}
\email{hsambale@math.uni-bielefeld.de}

\thanks{This research was supported by CRC 701.}
\begin{abstract}
	We show sharpened forms of the concentration of measure phenomenon centered at first order stochastic expansions. The bound are based
	on second order difference operators and second order derivatives. Applications to functions on the discrete cube and stochastic
	Hoeffding type expansions in mathematical statistics are studied as well as eigenvalue statistics in random matrix theory.
\end{abstract}
\date{\today}

\maketitle

\section{Introduction}

The concentration of measure phenomenon for product measures has been extensively studied in the past decades.
It was established by M. Talagrand in the 1990s \cite{T1}, \cite{T2}. 
Further research was done by S. Bobkov, M. Ledoux and others
\cite{L1}, \cite{B-G1}, \cite{B-G2}. For a comprehensive survey which
summarizes the central concentration of measure results up to the end of the 1990s see the monographs by M. Ledoux \cite{L2}, \cite{L3}, for a more recent one see \cite{B-L-M2}.

One of the basic results due to M. Talagrand are concentration inequalities for Lipschitz functions around their mean or
median. For instance, in discrete probability models, the product
probability space $(\Omega, \mathcal{A}, \mu) := \otimes_{i=1}^n (\Omega_i, \mathcal{A}_i, \mu_i)$ is typically equipped
with the Hamming distance $d(x,y) := \text{card} \{k = 1, \ldots, n \colon x_k \ne y_k \}$. A related approach, which is
essentially due to M. Ledoux \cite{L1}, makes use of certain ``difference operators''. That is, for any function $f \colon
\Omega \to \mathbb{R}$ in $L^2(\mu)$, set
\begin{equation}\label{DiskrDiff2}
\mathfrak{d}_i f(x) := \Big(\frac{1}{2} \int_{\Omega_i}(f(x)-f(x_1, \ldots, x_{i-1}, y_i, x_{i+1}, \ldots, x_n))^2 \mu_i(dy_i)\Big)^{1/2}
\end{equation}
and $\mathfrak{d} f := (\mathfrak{d}_1f, \ldots, \mathfrak{d}_nf)$.
%See Section 2 for a detailed description of the framework of the difference operators we use in this article.
A slight modification of \cite[Proposition 2.1]{B-G2} then yields

\begin{proposition}
\label{1.Ordn}
Let $(\Omega_i, \mathcal{A}_i, \mu_i)$ be probability spaces, and denote by $(\Omega, \mathcal{A}, \mu) := \otimes_{i=1}^n
(\Omega_i, \mathcal{A}_i, \mu_i)$ their product.
Moreover, let $f \colon \Omega \to \mathbb{R}$ be a bounded measurable function such that $\int f d\mu = 0$.
%Denote by $\mathfrak{d}$ the difference operator
%from \eqref{DiskrDiff2}, and
Assume that $|\mathfrak{d} f| \le 1$.
Then, for any $t \ge 0$ we have
$$\mu \left(\lvert f \rvert \ge t \right) \le 2 e^{-t^2/4}.$$
\end{proposition}

Note that the boundedness of $f$ is in fact a consequence of the condition $|\mathfrak{d} f| \le 1$ (see Section 2).
If we apply Proposition \ref{1.Ordn} to $1$-Lipschitz functions with respect to the Hamming distance, we recover the classical
concentration inequalities by M. Talagrand (cf. \cite{L1}). Similar results can be derived in the context of ``penalties'', which can be
regarded as generalizations of the Hamming distance \cite{L1}. In \cite{B-G2}, a generalized version of Proposition \ref{1.Ordn} is used for deriving
concentration inequalities for randomized sums.

The tail bounds we deduce in the present article are motivated as follows: consider a suitably normalized non-linear statistic which is stochastically non-degenerate and bounded in the limit. By general principles, it will have the same limit distribution as a stochastically bounded Gaussian chaos functional. If it is non-degenerate of order $2$ in the limit, it certainly has exponential tail decay.

As an example, consider a set of i.i.d. centered random variables $X_1, \ldots, X_n$ in $L^\infty$ (e.\,g. Rademacher variables). Then, a particularly simple case where Proposition \ref{1.Ordn} applies is the function $g(X) := \frac{1}{\sqrt{n}} \sum_{i=1}^n X_i$. A natural second order analogue of $g$ is the function $f(X) := \frac{1}{n} \sum_{i < j}X_iX_j$. However, in this case, evaluating the condition $|\mathfrak{d} f| \le 1$ and hence applying Proposition \ref{1.Ordn} does not lead to correct results. Indeed, $f$ is not a function of a Lipschitz class bounded in $n$.
%(In fact, one cannot even expect subgaussian tails for such functions, as the tails are known to be heavier, i.\,e. subexponential for large $t$.)
This motivates the use of second order differences instead. A further aspect can be observed if we do not assume the $X_i$ to be centered. In this case, we shall replace $f(X)$ by $Rf(X) := \frac{1}{n} \sum_{i < j}(X_i - \mathbb{E}X_i)(X_j - \mathbb{E}X_j)$. Comparing $Rf$ to $f$, we see that we have not only removed the expected value of $f$ but also a sort of ``linear term''.

Indeed, the notion of \emph{second} order concentration has two aspects which generalize these observations. First, it refers to the use of difference operators of second order. Second, it means that instead of
fluctuations of $f - \mathbb{E}f$ we will study
fluctuations of $f - \mathbb{E}f - f_1$, where $f_1$ is the first order term in the Hoeff\-ding decomposition of $f$. Let us briefly
recall the notion of Hoeffding decomposition, which was introduced in \cite{H}. Given a product probability
space $(\Omega, \mathcal{A}, \mu) := \otimes_{i=1}^n (\Omega_i, \mathcal{A}_i, \mu_i)$ and some function $f \in L^1(\mu)$, the
Hoeffding decomposition is the unique decomposition
\begin{align}
\label{Hoeffding}
f(x_1, \ldots, x_n) = \int f d\mu + \sum_{i=1}^{n} h_i(x_i) + \sum_{i<j} h_{ij}(x_i, x_j) + \ldots
= f_0 + f_1 + f_2 + \ldots + f_n
\end{align}
such that $\int h_{i_1 \ldots i_k}(x_{i_1}, \ldots, x_{i_k}) \mu_{i_j} (d x_{i_j}) = 0$ for all $k = 1, \ldots, n$, $1 \le i_1 < \ldots <
i_k \le n$ and $j \in \{1, \ldots, k\}$.
The sum $f_d$ is called the Hoeffding term of degree $d$ or simply $d$-th Hoeffding term of $f$. Note that for $f \in L^2(\mu)$
the $f_j, j \in {\mathbb N}_0$, form an orthogonal decomposition of $f$ in $L^2(\mu)$.

We now formulate our main results. For that, we need to introduce a notion of second order differences based on $\mathfrak{d}$. Indeed, for any function $f \colon \Omega \to \mathbb{R}$ in $L^2(\mu)$ and any $i \ne j$, set
\begin{align}
\mathfrak{d}_{ij} f(x) := &\Big(\frac{1}{4} \int_{\Omega_i}\int_{\Omega_j}(f(x)-f(x_1, \ldots, x_{i-1}, y_i, x_{i+1}, \ldots, x_n)\label{nablaiteriert}\\ &- f(x_1, \ldots, y_j, \ldots, x_n) + f(x_1, \ldots, y_i, \ldots, y_j, \ldots, x_n))^2 \mu_i(dy_i)\mu_j(dy_j)\Big)^{1/2}.\notag
\end{align}
In particular, we consider the following modified ``Hessian'' with respect to $\mathfrak{d}$:
\begin{align}
\label{Hessenabla}
(\mathfrak{d}^{(2)} f(X))_{ij} := \begin{cases} \mathfrak{d}_{ij}f(X), & i \neq j, \\ 0, & i = j. \end{cases}
\end{align}
For $x\in \mathbb R^n$ let $|x|$ denote its Euclidean norm, and for an $n \times n$ matrix $A = (a_{ij})_{ij}$ let
$\lVert A\rVert_{\text{HS}}$
denote its Hilbert--Schmidt norm given by $\lVert A \rVert_{\text{HS}} = (\sum_{i,j=1}^n |a_{ij}|^2)^{1/2}$.

\begin{theorem}
\label{zentral}
Let $(\Omega_i, \mathcal{A}_i, \mu_i)$ be probability spaces, and denote by $(\Omega, \mathcal{A}, \mu) := \otimes_{i=1}^n
(\Omega_i, \linebreak[2]\mathcal{A}_i, \mu_i)$ their product.
Moreover, let $f \colon \Omega \to \mathbb{R}$ be a bounded measurable function so that its Hoeffding decomposition
with respect to $\mu$ is given by
$f = \sum_{k=2}^{n} f_k.$
Assume that the conditions
\begin{equation}\label{bedingungen}
|\mathfrak{d}|\mathfrak{d} f|| \le 1\qquad \text{and}\qquad
\int \lVert \mathfrak{d}^{(2)} f \rVert_{\text{\emph{HS}}}^2 d\mu \le b^2
\end{equation}
are satisfied for some $b \ge 0$, where $\lVert \mathfrak{d}^{(2)} f \rVert_{\text{\emph{HS}}}$ denotes the Hilbert--Schmidt norm of $\mathfrak{d}^{(2)} f$.
Then, we have
$$\int \exp \left(\frac{1}{2(3 + b^2)} |f|\right) d\mu \le 2.$$
\end{theorem}

Note that by Chebychev's inequality, Theorem \ref{zentral} implies
$\mu(|f| \ge t) \le 2 e^{-ct}$
for all $t > 0$ and some constant $c = c(b^2)$ (in fact, $c = (2(3+b^2))^{-1}$). In other words, Theorem \ref{zentral} yields subexponential tails with an optimal exponent for large $t$ in accordance with the discussion above.

In Theorem \ref{zentral}, we have one condition assuming pointwise boundedness of second order-type differences and a second condition assuming boundedness in mean of the squared Hilbert--Schmidt norm of a suitable ``Hessian''. This mirrors the structure of Theorem 1.1 in S.\,G. Bobkov, G.\,P. Chistyakov and F. G\"{o}tze \cite{B-C-G}.
%In many situations (e.\,g. functions of Rademacher variables), the pointwise condition $|\mathfrak{d}|\mathfrak{d} f|| \le 1$ is the dominating one, and it can be seen as an analogue of the condition $|\mathfrak{d}f| \le 1$ in Proposition \ref{1.Ordn}.
%In many situations, the pointwise condition $|\mathfrak{d}|\mathfrak{d} f|| \le 1$ is the dominating one.
Some ways of explicitly evaluating the pointwise condition $|\mathfrak{d}|\mathfrak{d} f|| \le 1$ are given in Section 6. In particular, if for large $n$, the much stronger Hilbert--Schmidt norm is used, this will bound the constant $b$ in the second condition in \eqref{bedingungen}. In other situations (especially in differentiable settings, cf. Section 1.1), $|\mathfrak{d}|\mathfrak{d} f||$ may be bounded by operator-type norms of second order differences or derivatives.

Note that in general, the two conditions are incomparable: though the pointwise condition will often dominate, sometimes it may happen that the second one is more restrictive. An elementary example is given by the function $f(X) := \frac{1}{\sqrt{n}} (X_1X_2 + X_3X_4 + \ldots)$ for a set of independent Rademacher variables and $n$ even (here, we obviously have $|\mathfrak{d}|\mathfrak{d} f|| = 0$, and thus, the pointwise condition does not suffice in order to control the variance in the exponential estimates).
%A more advanced example of the second condition dominating the first one is provided in Example \ref{secondorderbeta}.
%The value of the latter constant as given
%by the bounds in Theorem \ref{zentral} and Corollary \ref{zentralfuerBernoulli} is not optimal, but optimizing it seems hard.

%Using Chebychev's inequality, Theorem \ref{zentral} and Corollary \ref{zentralfuerBernoulli} for instance imply the estimate
%$$\mu(|f| \ge t) \le 2 e^{-ct}$$
%for all $t > 0$ and some constant $c = c(b^2)$. The value of the latter constant as given
%by the bounds in Theorem \ref{zentral} and Corollary \ref{zentralfuerBernoulli} is not optimal, but optimizing it seems hard.
%It is possible to obtain a slightly better but still non-optimal constant from the proof of Theorem \ref{zentral}
%and Corollary \ref{zentralfuerBernoulli}.

For applications, we formulate a convenient ``hybrid'' bound extending the results from Theorem \ref{zentral} to functions with non-vanishing first order Hoeffding term (say, $f_1$). To this end, we need to provide that $f_1$ is of sufficiently small stochastic size. That is, in Theorem \ref{zentral}, let $f \colon \Omega \to
\mathbb{R}$ be a function in $L^1(\mu)$ with Hoeffding decomposition $f = \sum_{k=0}^{n} f_k$. Then, we denote by
\begin{equation}
\label{Projektion}
Rf := f - f_0 - f_1 = \sum_{k=2}^{n} f_k
\end{equation}
the projection of $f$ onto the space of the functions $f \in L^1(\mu)$ whose Hoeffding terms of orders 0 and 1 vanish. For convenience
we shall assume that the expected value $f_0$ of $f$ vanishes. In order to obtain a result similar to Theorem \ref{zentral},
%for instance
%$$\int e^{c |f|} d\mu \le \int e^{c (|f_1| + |Rf|)} d\mu \le 2$$
%for some constant $c > 0$,
we add conditions ensuring $f_1 = \mathcal{O}_P(1)$ (cf. Proposition \ref{1.Ordn}). The result is the following theorem:

\begin{theorem}
	\label{mit Hoeffding1}
	Let $(\Omega_i, \mathcal{A}_i, \mu_i)$ be probability spaces, and denote by $(\Omega, \mathcal{A}, \mu) := \otimes_{i=1}^n
	(\Omega_i, \linebreak[2]\mathcal{A}_i, \mu_i)$ their product. Moreover, let $f \colon \Omega \to \mathbb{R}$ be a bounded measurable function
	such that its Hoeffding decomposition with respect to $\mu$ is given by
	$f = f_1 + \sum_{k=2}^{n} f_k = f_1 + Rf$.
	(In particular, we have $\mathbb{E}f = 0$.) Suppose that $|\mathfrak{d} f_1| \le b_0$ for some $b_0 \ge 0$ and that the conditions
	$$|\mathfrak{d}|\mathfrak{d} Rf|| \le 1\qquad \text{and}\qquad
	\int \lVert \mathfrak{d}^{(2)} f \rVert_{\text{\emph{HS}}}^2 d\mu \le b^2$$
	for some $b \ge 0$ are satisfied.
	Then, we have
	$$\int \exp\left(\frac{1}{12 + 4 b^2 + 7 b_0} |f|\right) d\mu \le 2.$$
\end{theorem}

%Similar as in Corollary \ref{zentralfuerBernoulli}, it is possible to improve the constants $c_1$ and $c_2$
%if all the underlying measures are Bernoulli distributions. We skip details at this point.

\subsubsection{Discussion of Related Inequalities}

Hoeffding decompositions have been studied in particular in the context of $U$-statistics, that is, statistics of the form
$
U_n(h) = \frac{(n-m)!}{n!} \sum_{i_1 \ne \ldots \ne i_m} h(X_{i_1}, \ldots, X_{i_m})
$
for a sequence of i.i.d. random variables $(X_i)_{i \in \mathbb{N}}$, a measurable kernel function $h$ on $\mathbb{R}^m$ and
natural numbers $n, m$ such that $n \ge m$. A $U$-statistic is called completely degenerate (or canonical) if its Hoeffding
decomposition consists of a single term only.
There are a lot of results on the distributional properties of $U$-statistics. A partial overview is given in the monograph by V. de la Pe\~{n}a and E. Gin\'{e} \cite{D-G}. In particular, there are many inequalities describing their tail behavior starting
with Hoeffding's inequalities. That is, for $U$-statistics like $U_n(h)$ introduced above, we have
$P (U_n(h) > t) \le \exp (- [n/m]t^2/(2 M^2))$
if the function $h \colon \mathbb{R}^m \to \mathbb{R}$ is bounded by some universal constant $M$ and satisfies $\mathbb{E}
h(X_1, \ldots, X_m) = 0$. Further exponential inequalities for completely degenerate $U$-statistics have been proved by M. Arconas and E. Gin\'{e} \cite{A-G} as well as P. Major \cite{M}. These inequalities typically depend on the order $m$,
the second moment $\sigma^2$ and some bound $M$ of the kernel $h$ only.

Finally, let us mention that in a subsequent paper together with S.\,G. Bobkov  \cite{B-G-S}, we have extended some of the results of the present paper to arbitrary higher orders. However, the methodology is quite different. For instance, in the present paper our arguments are mainly based on modified logarithmic Sobolev inequalities and exponential inequalities which follow from them. By contrast, the main tool in \cite{B-G-S} is a recursion inequality for the $L^p$-norms of $f$ and its higher order differences (or derivatives) for any $p \ge 2$, which in turn does not appear in the present paper.

\subsection{Differentiable Functions}

In arbitrary product spaces, the usual notion of differentiation is not available, which is why we need to work with difference operators as a kind of substitute.
However, if we do consider differentiable settings, it seems natural to use the ordinary gradient $\nabla$ instead. Therefore, we now complement our main theorems by results valid in differentiable settings.

Indeed, it is possible to formulate a result similar to Theorem \ref{zentral} for probability measures on $\mathbb{R}^n$ which satisfy a logarithmic Sobolev inequality. Note that this situation has already been sketched in \cite{B-C-G} (see Remark 5.3 there). In the present paper, we work out these ideas in detail and add some further material, including a version with an additional ``linear'' term similar to Theorem \ref{mit Hoeffding1} and some applications. Let us first recall some basic notions.

Let $G \subset \mathbb{R}^n$ be some open set, and let $\mu$ be a probability measure on $(G, \mathcal{B}(G))$. Then, $\mu$ satisfies
a \emph{Poincar\'{e} inequality} with constant $\sigma^2 > 0$ if for all locally Lipschitz functions $f \colon G \to \mathbb{R}$
\begin{equation}
\label{PI}
\text{Var}_\mu (f) \le \sigma^2 \int_G |\nabla f|^2 d\mu,
\end{equation}
where $\text{Var}_\mu (f) = \int f^2 d\mu - (\int f d\mu)^2$ and $|\nabla f|$ denotes the Euclidean norm of the usual gradient.
Another type of functional inequality for probability measures $\mu$ on $(G, \mathcal{B}(G))$ is given by the \emph{logarithmic
Sobolev inequality}.
That is, $\mu$ satisfies a logarithmic Sobolev inequality with (Sobolev) constant $\sigma^2 > 0$ if for all locally Lipschitz functions
$f \colon G \to \mathbb{R}$
\begin{equation}
\label{LSI}
\text{Ent}_\mu (f^2) \le 2 \sigma^2 \int_G |\nabla f|^2 d\mu,
\end{equation}
where $\text{Ent}_\mu (f^2) = \int f^2 \log f^2 d\mu - \int f^2 d\mu \log\int f^2 d\mu$ (see Section 3). Logarithmic Sobolev
inequalities are stronger than Poincar\'{e} inequalities. For instance, if $\mu$ satisfies a logarithmic Sobolev inequality with
constant $\sigma^2$, it also satisfies a Poincar\'{e} inequality with the same constant $\sigma^2$.

We now have the following result:

\begin{theorem}
\label{kontinuierlich}
Let $G \subset \mathbb{R}^n$ be some open set, and let $\mu$ be a probability measure on $(G, \mathcal{B}(G))$ which satisfies a
logarithmic Sobolev inequality with constant $\sigma^2 > 0$. Let $f \colon G \to \mathbb{R}$ be a $\mathcal{C}^2$-smooth function
such that $f \in L^1(\mu)$ and $\partial_i f \in L^1(\mu)$ for all $ i = 1, \ldots, n$, where $\partial_i f$ denotes the $i$-th
partial derivative of $f$. Assume that
$\int_G f d\mu = 0$ and $\int_G \partial_i f d\mu = 0$ for all $i = 1, \ldots, n$.
Moreover, assume that
$$\lVert f''(x) \rVert_\text{\emph{Op}} \le 1\enskip \text{for all $x \in G$}\qquad \text{and}\qquad
\int_G \lVert f'' \rVert_\text{\emph{HS}}^2 d\mu \le b^2$$
for some $b \ge 0$, where $f''$ denotes the Hessian of $f$ and $\lVert f'' \rVert_\text{\emph{Op}}$, $\lVert f'' \rVert_\text{\emph{HS}}$ denote its operator and Hilbert--Schmidt norms, respectively.
Then, the following inequality holds:
$$\int_G \exp \left(\frac{1}{2 \sigma^2(1 + b^2)} |f|\right) d\mu \le 2.$$
\end{theorem}

Note that unlike in Theorem \ref{zentral}, we do not need to require $\mu$ to be a product measure. Given any function $f \in
\mathcal{C}^2(G)$ such that $f \in L^1(\mu)$ and $\partial_i f \in L^1(\mu)$ for all $ i = 1, \ldots, n$, we may modify
$f$ to remove a ``linear'' term by considering
\begin{equation}
\label{corrected}
Rf(x) = f(x) - \mu[f] - \sum_{i=1}^{n} \mu[\partial_i f] (x_i - \mu[x_i]),
\end{equation}
where $\mu[h] = \int_G h d\mu$ for any function $h \in L^1(\mu)$. $Rf$ represents a centered function with centered derivatives.

Similarly to Theorem \ref{mit Hoeffding1}, we may allow non-vanishing integrals $\mu[\partial_i f]$ in Theorem \ref{kontinuierlich}
if they are of sufficiently small size. In detail:

\begin{theorem}
	\label{kontinuierlich+1.Ordn}
	Let $G \subset \mathbb{R}^n$ be some open set, and let $\mu$ be a probability measure on $(G, \mathcal{B}(G))$ which satisfies a
	logarithmic Sobolev inequality with constant $\sigma^2 > 0$. Let $f \colon G \to \mathbb{R}$ be a $\mathcal{C}^2$-smooth function
	such that $f \in L^1(\mu)$ and $\partial_i f \in L^1(\mu)$ for all $ i = 1, \ldots, n$, where $\partial_i f$ denotes the $i$-th
	partial derivative of $f$. Assume that
	$\int_G f d\mu = 0$ and $\sum_{i=1}^{n}(\int_G \partial_i f d\mu)^2 \le \sigma^2b_0^2$
	for some $b_0 \ge 0$. Moreover, assume that
	$$\lVert f''(x) \rVert_\text{\emph{Op}} \le 1\enskip \text{for all $x \in G$}\qquad \text{and}\qquad
	\int_G \lVert f''(x) \rVert_\text{\emph{HS}}^2 d\mu \le b^2$$
	for some $b \ge 0$, where $f''$ denotes the Hessian of $f$ and $\lVert f'' \rVert_\text{\emph{Op}}$, $\lVert f'' \rVert_\text{\emph{HS}}$ denote its operator and Hilbert--Schmidt norms, respectively.
	Then, we have
	$$\int_G \exp \left(\frac{1}{\sigma^2(4 + 4 b^2 + 5 b_0)} |f|\right) d\mu \le 2.$$
\end{theorem}

\subsubsection{Discussion of Related Inequalities}

We shall compare our results to a measure concentration result for functions on the $n$-sphere which are orthogonal to linear
functions, see S.\,G. Bobkov, G.\,P. Chistyakov and F. G\"{o}tze \cite{B-C-G}. In this context, Theorem \ref{zentral} can be
regarded as a ``discrete'' analogue of the latter result. Note that in particular, it covers the case of the discrete hypercube
$\{\pm1 \}^n$ equipped with the uniform distribution. Theorem \ref{kontinuierlich} may then be seen as an intermediate between
Theorem \ref{zentral} and the bounds in \cite{B-C-G}. Indeed, if in Theorem \ref{kontinuierlich} $\mu$ is the standard Gaussian
measure, the condition $\int \partial_i f d\mu = 0$ for all $i$ is satisfied if we require orthogonality
to all linear functions (by partial integration). The idea of sharpening concentration inequalities for Gaussian and related
measures by requiring orthogonality to linear functions also appears in \cite{CE-F-M}.

We would moreover like to mention the results by R. Adamczak and P. Wolff \cite{A-W}. They study the tail behavior 
of differentiable functions. Requiring certain Sobolev-type inequalities or subgaussian tail conditions, they derive exponential
inequalities for functions with bounded higher-order derivatives (evaluated in terms of some tensor-product matrix norms). In
comparison, our paper has a stronger emphasis on discrete models and difference operators with a focus on functions structured
by Hoeffding expansions of vanishing first order or, in differentiable cases as in Theorem \ref{kontinuierlich}, functions from
which we remove a kind of ``linear term''.

\subsection{Outline}

The main tools we use in this article will be introduced in Sections 2 and 3. This includes some basic facts about difference operators,
Hoeffding decompositions and modified logarithmic Sobolev inequalities. The proofs of our main theorems for product measures will be
given in Sections 4 and 5. Here, we will first derive exponential inequalities based on modified Sobolev inequalities. After that, second order differences will be invoked by making use of certain ``harmonic analysis'' arguments on the symmetric group established in Section 2. The proof of Theorem \ref{zentral} then follows as an easy combination of both chains
of arguments.

In Section 6, we discuss how to evaluate the second order conditions from Theorem \ref{zentral}. In particular, we give a
reformulation of Theorem \ref{zentral} which involves conditions which may be easier to apply. We also apply our results to functions
of independent Rademacher variables.

The differentiable case will be discussed in Section 7. Here we need to modify some of the arguments from the proof of
Theorems \ref{zentral} and \ref{mit Hoeffding1}. Together with a simple application of the Poincar\'{e} inequality, this will lead
us to the proof of Theorems \ref{kontinuierlich} and~\ref{kontinuierlich+1.Ordn}.

Finally, Section 8 presents a number of examples for functions of independent random variables as well as in differentiable settings.

A prior version of these results is based on the Ph.D. thesis of the second author \cite{S}.

\textbf{Acknowledgements.}

We wish to thank Sergey Bobkov for important suggestions concerning	modified log-Sobolev inequalities which helped to improve this paper
and Holger K\"{o}sters and Arthur Sinulis for many fruitful discussions.

\section{Difference Operators}

Let $(\Omega_1, \mathcal{A}_1), \ldots, (\Omega_n, \mathcal{A}_n)$ be measurable spaces, and denote by $(\Omega, \mathcal{A})$ their product space.
Similarly to \cite{B-G1}, we study (difference) operators $\Gamma$ on the space of
the bounded measurable real-valued functions on $(\Omega, \mathcal{A})$ such that the following two conditions hold (in particular, no sort of ``Leibniz rule'' is required):

\begin{condition}
\begin{enumerate} [(i)]
\item For any bounded measurable function $f \colon \Omega \to \mathbb{R}$, $\Gamma f = (\Gamma_1 f, \ldots,\linebreak[2] \Gamma_n f)
\colon \Omega \to \mathbb{R}^n$ is a measurable function with values in $\mathbb{R}^n$. We often call $\Gamma$ a
\emph{gradient operator} or simply \emph{gradient}.
\item For all $i = 1, \ldots, n$, all $a > 0$, $b \in \mathbb{R}$ and any bounded measurable real-valued function $f$, we have
$|\Gamma_i (af + b)| = a |\Gamma_i f|$.
\end{enumerate}
\label{DiskrDiffallgem}
\end{condition}

In addition to the ``$L^2$ difference operator'' $\mathfrak{d}$ in \eqref{DiskrDiff2}, we need a
difference operator adapted to the Hoeffding decomposition. Indeed, for any function $f \colon \Omega \to \mathbb{R}$ in $L^1(\mu)$, let
\begin{equation}
\label{DiskrDiffFormel}
\mathfrak{D}_if(x) := f(x) - \int_{\Omega_i}f(x_1, \ldots, x_{i-1}, y_i, x_{i+1}, \ldots, x_n)\mu_i(d y_i)
\end{equation}
and $\mathfrak{D}f := (\mathfrak{D}_1f, \ldots, \mathfrak{D}_nf)$. Higher order differences are defined by iteration, e.\,g.
$\mathfrak{D}_{ij}f := \mathfrak{D}_i(\mathfrak{D}_jf)$ for
$1 \le i, j \le n$. As in \eqref{Hessenabla}, we then define a modified ``Hessian'' with respect to $\mathfrak{D}$ by
\begin{align}
\label{Hessediskret}
(\mathfrak{D}^{(2)}f(x))_{ij} := \begin{cases} \mathfrak{D}_{ij}f(x), & i \neq j, \\ 0, & i = j. \end{cases}
\end{align}

The difference operator $\mathfrak{D}$ is closely related to the Hoeffding decomposition \eqref{Hoeffding}. In essence, proving
\eqref{Hoeffding} is based on the identity $\mathbb{E}_i + \mathfrak{D}_i = Id$ with $\mathfrak{D}_i$ as in \eqref{DiskrDiffFormel}. We finally get
$h_{i_1 \ldots i_k}(X_{i_1}, \ldots, X_{i_k}) = (\prod_{j \notin \{i_1, \ldots i_k\}} \mathbb{E}_j \prod_{l \in \{i_1, \ldots i_k\}}
\mathfrak{D}_l) f(X_1, \ldots, X_n)$.

%Clearly, the difference operators $\mathfrak{d}$, $\mathfrak{D}$ and $\mathfrak{d}^+$ from \eqref{DiskrDiff2}, \eqref{DiskrDiffFormel}
%and \eqref{DiskrDiff3}
%satisfy Conditions \ref{DiskrDiffallgem}. Note that here we assume the spaces $(\Omega_i, \mathcal{A}_i)$ to be
%endorsed with probability measures $\mu_i$, whose product measure we denote by $\mu$.
Let us collect some elementary facts about
the difference operators $\mathfrak{d}$ and $\mathfrak{D}$. In the following assume that $X_1, \ldots, X_n$ is a sequence of independent random variables on
some probability space $(\Omega', \mathcal{A}', P)$ with distributions $\mu_1, \ldots, \mu_n$ respectively. As we will see, introducing
random variables sometimes facilitates notation:

\begin{remark}
	\label{DiskrDiff}
	\hspace{2em}
	\begin{enumerate}
	\item If $\mu_i = \frac{1}{2} \delta_{+1} + \frac{1}{2} \delta_{-1}$ for all $i = 1, \ldots, n$, we have
	$\mathfrak{D}_if(X) = \frac{1}{2}(f(X)-f(\sigma_iX))$,
	where $X = (X_1, \ldots, X_n)$ and $\sigma_iX := (X_1, \ldots, -X_i, \ldots, X_n)$. Moreover, note that $\mathfrak{d}_i f =
	|\mathfrak{D}_if|$.
	%and $\mathfrak{d}^+_i f = (\mathfrak{D}_i f)_+$.
	
		\item For any function $f(X) \in L^1(P)$, we have
		$\mathfrak{D}_if(X) = f(X) - \mathbb{E}_if(X)$
		or (in short) $\mathfrak{D}_i = Id - \mathbb{E}_i$. Here, $Id$ denotes the identity and $\mathbb{E}_i$ taking the expectation with respect
		to $X_i$.
		
		\item Let $f(X) \in L^2(P)$, and let $\bar{X}_1, \ldots, \bar{X}_n$ be a set
		of independent copies of the random variables $X_1, \ldots, X_n$. Set $T_if := f(X_1, \ldots, X_{i-1}, \bar{X}_i,\linebreak[2] X_{i+1},
		\ldots, X_n)$ for any function $f(X_1, \ldots, X_n)$. Then, we have
		$$
		\mathfrak{d}_if(X) = (\frac{1}{2}\bar{\mathbb{E}}_i(f(X) - T_if(X))^2)^{1/2}.
		$$
		Here, $\bar{\mathbb{E}}_i$ denotes the expectation with respect to $\bar{X}_i$. By independence, if $\mathbb{E}_i$ denotes the expectation with respect to $X_i$ we can rewrite
		\begin{align}
		\label{nabla2}
		\mathfrak{d}_if(X) &= \Big(\frac{1}{2}\big((f(X) - \mathbb{E}_if(X))^2 + \mathbb{E}_i(f(X) - \mathbb{E}_if(X))^2\big)\Big)^{1/2}\notag\\
		&= \Big(\frac{1}{2}\big((\mathfrak{D}_if(X))^2 + \mathbb{E}_i(\mathfrak{D}_if(X))^2\big)\Big)^{1/2}.
		\end{align}
		
		\item Setting $T_{ij} = T_i \circ T_j$, second order analogues of the formulas for $\mathfrak{d}_i$ are given by
		\begin{align}
		\label{nabla3}
		\mathfrak{d}_{ij} f(X) &= \Big(\frac{1}{4}\bar{\mathbb{E}}_{ij}(f(X) - T_if(X) - T_jf(X) + T_{ij}f(X))^2\Big)^{1/2},\\
		\label{nablaiteriert2}
		\mathfrak{d}_{ij} f &= \Big(\frac{1}{4}\big((\mathfrak{D}_{ij}f)^2 + \mathbb{E}_i(\mathfrak{D}_{ij}f)^2 + \mathbb{E}_j(\mathfrak{D}_{ij}f)^2
		+ \mathbb{E}_{ij}(\mathfrak{D}_{ij}f)^2\big)\Big)^{1/2}
		\end{align}
		for any $i \ne j$. Here, $\bar{\mathbb{E}}_{ij}$ means taking the expectation with respect to
		$\bar{X}_i$ and $\bar{X}_j$, and
		$\mathbb{E}_{ij}$ means taking the expectation with respect to $X_i$ and $X_j$.
		
%		\item Similarly, we have
%		$$\mathfrak{d}^+_if(X) = \Big(\frac{1}{2}\bar{\mathbb{E}}_i(f(X) - T_if(X))_+^2\Big)^{1/2}$$
%		for any $f(X) \in L^2(P)$ as well as
%		$$\mathfrak{D}_i f(X) = \bar{\mathbb{E}}_i(f(X) - T_if(X))$$
%		for any $f(X) \in L^1(P)$.
	\end{enumerate}
\end{remark}

By induction over $n$, $f$ is bounded if and only if $|\mathfrak{D}f|$ is bounded. Using \eqref{nabla2}, the same holds for $|\mathfrak{d} f|$
instead of $|\mathfrak{D}f|$. Moreover, it follows immediately from \eqref{nablaiteriert2} that
\begin{equation}\label{L^2Normengleich}
\int \lVert \mathfrak{d}^{(2)} f \rVert_{\textrm{HS}}^2 d\mu = \int \lVert \mathfrak{D}^{(2)} f \rVert_{\textrm{HS}}^2 d\mu,
\end{equation}
which will turn out to be an important identity in our proof.

For some kind of ``harmonic'' analysis arguments on the symmetric group, we shall need a specific second order operator we would call
``Laplacian''. Since in our discrete setting $\mathfrak{D}_{ii} = \mathfrak{D}_i$ for all $i$, this cannot be $\mathfrak{L} = \sum_i
\mathfrak{D}_{ii}$. Instead, we define
\begin{equation}
\mathfrak{L} := \sum_{i \neq j} \mathfrak{D}_{ij}.
\label{Laplace}
\end{equation}

Calling \eqref{Laplace} a Laplacian is justified for several reasons. First of all, \eqref{Laplace} enjoys similar properties
with respect to scalar products in function spaces (see Lemma \ref{selbstadj} below) compared to the classical Euclidean or spherical
Laplacian. Moreover, if we assume $\mu_i \equiv \mu_1$ for all $i$ in Example \ref{DiskrDiff}, that is for functions of i.i.d. random
variables, the Laplacian \eqref{Laplace} is invariant under permutations, i.\,e.
$\mathfrak{L} f(x) = \mathfrak{L} f(\pi (x))$
for any $\mu$-integrable function $f$ on $\mathbb{R}^n$ and any permutation $\pi$
of $\{1, 2, \ldots, n\}$. As usual, here we set $f(\pi(x)) = f(x_{\pi^{-1}(1)}, \ldots,\linebreak[2] x_{\pi^{-1}(n)})$. This
may be regarded as a discrete analogue of the rotational invariance of the usual Laplacian.

Relating the Hoeffding decomposition to the Laplacian $\mathfrak{L}$ yields the following result:

\begin{theorem}
\label{diagonal}
Let $(\Omega_i, \mathcal{A}_i, \mu_i)$ be probability spaces, and denote by $(\Omega, \mathcal{A}, \mu) := \otimes_{i=1}^n
(\Omega_i,\linebreak[2] \mathcal{A}_i, \mu_i)$ their product. Moreover, let $f$ be some function in $L^1(\mu)$ with Hoeffding decomposition $f = \sum_{d=0}^n f_d$.
Then, we have
\begin{equation*}
\mathfrak{L} f_d = (d)_2 f_d.
\end{equation*}
Here, $\mathfrak{L}$ is the Laplacian as introduced in \eqref{Laplace}, and we write $(d)_2 = d(d-1)$. Thus, the $d$-th Hoeffding term is an
eigenfunction of $\mathfrak{L}$ with eigenvalue $(d)_2$.
\end{theorem}

Consequently, there is an orthogonal decomposition of $L^2$-functions $f$ on which the Laplacian operates diagonally.

\begin{proof}
Write $f_d(x_1, \ldots, x_n) = \sum_{i_1< \ldots < i_d} h_{i_1 \ldots i_d}(x_{i_1}, \ldots, x_{i_d})$ as in \eqref{Hoeffding}.
Fix $i_1 < \ldots < i_d$. Then, we get
$$\int h_{i_1 \ldots i_d}(x_{i_1} \ldots, x_{i_d}) \mu_i (d x_i) = \begin{cases} 0, & i \in \{i_1, \ldots, i_d\}, \\
h_{i_1 \ldots i_d}(x_{i_1}, \ldots, x_{i_d}), & i \notin \{i_1, \ldots, i_d\}. \end{cases}$$
Therefore, we have
\begin{align}
\label{D_i f_d}
\mathfrak{D}_if_d(x_1, \ldots, x_n) &= \sum_{\substack{i_1< \ldots < i_d \\ i \in \{i_1, \ldots, i_d\}}} h_{i_1 \ldots i_d}(x_{i_1},
\ldots, x_{i_d}),\\
\label{D_ij f_d}
\mathfrak{D}_{ij}f_d(x_1, \ldots, x_n) &= \sum_{\substack{i_1< \ldots < i_d \\ i, j \in \{i_1, \ldots, i_d\}}} h_{i_1 \ldots i_d}(x_{i_1},
\ldots, x_{i_d}).
\end{align}
Hence it remains to check how often each term $h_{i_1 \ldots i_d}(x_{i_1}, \ldots, x_{i_d})$ appears in $\mathfrak{L} f_d =\linebreak[2] \sum_{i \neq j}
\mathfrak{D}_{ij} f_d$. As we just saw, each pair $i \neq j$ such that $i, j \in \{i_1, \ldots, i_d\}$ replicates the summand
$h_{i_1 \ldots i_d}(x_{i_1}, \ldots, x_{i_d})$ precisely once. As there are $d (d-1) = (d)_2$ such pairs, we arrive at the result.
\end{proof}

In fact, there are at least two larger families of difference operators which satisfy similar ``invariance properties'' with respect to the
symmetric group and the Hoeffding decomposition.
One family of this type can be defined via
$\mathfrak{L}_1 := \sum_i \mathfrak{D}_i$, $\mathfrak{L}_2 := \mathfrak{L}_1^2$ and more generally $\mathfrak{L}_k
:= \mathfrak{L}_1^k$ for any $k \in \{1, 2, \ldots, n\}$. Another one is given by
$\mathfrak{L}^*_k := \sum_{i_1 \ne i_2 \ne \ldots \ne i_k } \mathfrak{D}_{i_1} \ldots \mathfrak{D}_{i_k}$
for any $k \in \{1, 2, \ldots, n\}$. It is possible to relate these two
families to each other by representing the $\mathfrak{L}^*_k$ as polynomials in $\mathfrak{L}_1$, e.\,g. we have $\mathfrak{L}^*_2 =
\mathfrak{L}_1^2 - \mathfrak{L}_1$.

As in the proof of Theorem \ref{diagonal}, simple combinatorial arguments show that all the
$\mathfrak{L}_k$ and $\mathfrak{L}^*_k$ operate diagonally on the Hoeffding decomposition. In case of the $\mathfrak{L}^*_k$,
the eigenvalues of the Hoeffding terms of order up to $k-1$ are $0$.

In particular, with $\mathfrak{L}$ as in \eqref{Laplace}, we see that we have $\mathfrak{L} = \mathfrak{L}^*_2$. In other words, $\mathfrak{L}$ is the
second order difference invariant operator which annihilates the Hoeffding terms up to first order. This is in accordance
with our basic concept of second order concentration.

%It would be interesting to study concentration of higher order for functions of independent random variables with the help of the
%operators $\mathfrak{L}^*_k$, but it seems that this will get more involved than the second order case and needs a different set of technical
%tools for proving concentration. We intend to return to this question in the future.

\section{Modified Logarithmic Sobolev Inequalities and Exponential Inequalities}

Let $\mu$ be a probability measure on some measurable space
$(\Omega, \mathcal{A})$ and $g \colon \Omega \to [0, \infty)$ a measurable function. Then, we define the entropy of $g$ with respect to
$\mu$ by
$\text{Ent}(g) := \text{Ent}_\mu(g) := \int g \log g d\mu - \int g d\mu \log \int g d\mu$.
Here, we set $\text{Ent}(g) := \infty$ if any of the integrals involved does not exist. A natural condition for existence of entropy is whether 
the integral of $g \log(1+g)$ is finite or not. It is well-known that by Jensen's inequality, we have $\text{Ent}(g) \in [0, \infty]$.
As a modification of the usual logarithmic Sobolev inequality, we now define

\begin{definition}
\label{DLSU}
Let $\mu$ be a probability measure on some measurable space $(\Omega, \mathcal{A})$, and let $\Gamma$ be a difference operator on this
space satisfying Conditions \ref{DiskrDiffallgem}. Then, $\mu$ satisfies a modified logarithmic Sobolev inequality with constant
$\sigma^2 > 0$ with respect to $\Gamma$ if for any bounded measurable function $f \colon \Omega \to \mathbb{R}$
\begin{equation}
\label{mLSI}
\text{\emph{Ent}}(e^f) \le \frac{\sigma^2}{2} \int |\Gamma f|^2 e^f d\mu.
\end{equation}
Here, $|\Gamma f|$ denotes the Euclidean norm of the gradient $\Gamma f$.
\end{definition}

This definition goes back to \cite{B-G1}, where it is called $\text{LSI}_{\sigma^2}$. The term ``modified logarithmic Sobolev inequality''
is due to \cite[Chapter 5.3]{L3}, where other modifications of logarithmic Sobolev inequalities are discussed as well.
The difference between the usual form of the LSI and the modified one in \eqref{mLSI} is motivated by the fact that difference
operators do not necessarily satisfy any sort of chain rule. The number $\sigma^2 > 0$ is also called \emph{Sobolev
constant}. When using $\sigma$ instead of $\sigma^2$ itself, we will always assume it to be positive.

We will use Definition \ref{DLSU} with $\Gamma = \mathfrak{d}$. Note that setting
$\Gamma = \mathfrak{D}$
would be too restrictive since in this case, only discrete probability measures with a finite number of atoms would have a chance to
fulfill a modified LSI of type \eqref{mLSI}. By contrast, in case of $\mathfrak{d}$ we have the following:

\begin{proposition}
\label{DLSUfueralle}
Let $\mu$ be any probability measure on some measurable space $(\Omega, \mathcal{A})$. Then, $\mu$ satisfies the modified LSI
\eqref{mLSI} with Sobolev constant $\sigma^2 = 2$ with respect to the gradient operator $\mathfrak{d}$ from \eqref{DiskrDiff2}.
\end{proposition}

\begin{proof}
This is due to \cite{B-G2} and essentially based on \cite{L1}. For the reader's convenience we include a sketch of its proof here. First, we apply Jensen's inequality to get
\begin{align*}
\text{Ent}_\mu(e^g) \le \text{Cov}_\mu(g,e^g)
&= \frac{1}{2} \iint (g(x)-g(y))(e^{g(x)}-e^{g(y)}) \mu(dx) \mu(dy)\displaybreak[2]\\
&\le \frac{1}{4} \iint (g(x)-g(y))^2(e^{g(x)}+e^{g(y)}) \mu(dx) \mu(dy)
= \int |\mathfrak{d} g|^2 e^g d\mu.
\end{align*}
Here $g$ is any real-valued measurable function on $\Omega$ such that the integrals involved are finite, and the next-to-last step
uses the elementary estimate $(a-b)(e^a-e^b) \le \frac{1}{2} (a-b)^2(e^a+e^b)$ for all $a, b \in \mathbb{R}$. However, this means
that $\mu$ satisfies the modified LSI \eqref{mLSI} with Sobolev constant $\sigma^2 = 2$.
\end{proof}

If we especially consider two-point measures, the Sobolev constant can still be improved a little:

\begin{proposition}
\label{DLSUfuerBernoulli}
Let $\mu = p \delta_{+1} + (1-p) \delta_{-1}$ for some $p \in (0,1)$, where $\delta_x$ denotes the Dirac measure in $x \in
\mathbb{R}$. Then, $\mu$ satisfies the modified LSI \eqref{mLSI} with Sobolev constant $\sigma^2 = 1$ with respect to $\mathfrak{d}$ as in
\eqref{DiskrDiff2}.
\end{proposition}

This is again due to \cite{B-G2}, and we omit the proof here. It is easy to verify that for instance in case of $p = \frac{1}{2}$, this
constant is optimal.

From Propositions \ref{DLSUfueralle} and \ref{DLSUfuerBernoulli}, we can easily go on to product spaces by the following tensorization
property which goes back to \cite{L1}:

\begin{lemma}
\label{DLSU-Prod}
For all $i = 1, \ldots, n$, let $(\Omega_i, \mathcal{A}_i)$ be measurable spaces equipped with probability measures $\mu_i$ each
satisfying the modified LSI \eqref{mLSI} with Sobolev constants $\sigma_i^2 > 0$ with respect to $\mathfrak{d}$ as in \eqref{DiskrDiff2}.
Then, the product measure $\mu_1 \otimes \ldots \otimes \mu_n$ on $(\Omega_1 \times \ldots \times \Omega_n, \mathcal{A}_1 \otimes
\ldots \otimes \mathcal{A}_n)$ also satisfies the modified LSI \eqref{mLSI} with Sobolev constant $\sigma^2 = \max_{i = 1, \ldots, n}
\sigma_i^2$ with respect to $\mathfrak{d}$.
\end{lemma}

As in the case the usual logarithmic Sobolev inequality, this is a consequence of the subadditivity (or tensorization) property of the entropy functional
together with the additivity property of the gradient operator $\mathfrak{d}$. Therefore, Propositions \ref{DLSUfueralle} and \ref{DLSUfuerBernoulli}
naturally extend to product measures.

\section{Exponential Inequalities}

In this section, we derive exponential moment inequalities for functions of independent random variables.
Consider any probability measure
on some measurable space $(\Omega, \mathcal{A})$ which satisfies the modified LSI \eqref{mLSI} with Sobolev
constant $\sigma^2 > 0$ with respect to $\mathfrak{d}$.
In \cite{B-G1}, it was proved that for all bounded
measurable functions $f \colon \Omega \to \mathbb{R}$ such that $\int f d\mu = 0$, we have
\begin{equation}
\int e^f d\mu \leq \int e^{\sigma^2 |\mathfrak{d} f|^2} d\mu.
\label{A'}
\end{equation}
The proof of \eqref{A'} is similar to the proof of inequality \eqref{B-G1-Ungl} which
will be sketched in the proof of Lemma \ref{ExpUnglVorstufe}.

In addition to \eqref{A'}, we need a second inequality of the form
$$\int e^{t u^2} d\mu \leq \exp\Big(c(t) \int u^2 d\mu\Big)$$
for small $t$ and some constant $c$ depending on $t$. An inequality of the desired form due to
\cite{A-M-S} is known if the
underlying gradient operator satisfies the chain rule (cf. \eqref{B} in Section 7). Here, the main argument for which the chain
rule is needed is as follows: let $\nabla$ denote the usual gradient and $|\nabla f|$ its Euclidean norm. Then, if
we assume $|\nabla f| \leq 1$, we immediately get
$|\nabla f^2| = 2 |f| |\nabla f| \leq 2 |f|$.
However, if we replace $\nabla$ by the $L^2$-difference operator $\mathfrak{d}$ from \eqref{DiskrDiff2},
such an inequality does not hold.

This desirable property is restored by switching to yet another difference operator which we denote by $\mathfrak{d}^+$. In detail,
\begin{equation}
\label{DiskrDiff3}
\mathfrak{d}^+_i f(x) := \Big(\frac{1}{2} \int_{\Omega_i}(f(x)-f(x_1, \ldots, x_{i-1}, y_i, x_{i+1}, \ldots, x_n))_+^2 \mu_i(dy_i)\Big)^{1/2}.
\end{equation}
Here, $f \colon \Omega \to \mathbb{R}$ is any function in $L^2(\mu)$, and $g_+ := \max(g, 0)$ denotes the positive part of any real-valued function $g$. As always, $\mathfrak{d}^+f = (\mathfrak{d}_1^+f, \ldots, \mathfrak{d}_n^+f)$.

Let $f \colon \Omega \to \mathbb{R}$ be any measurable function on some probability
space $(\Omega, \mathcal{A}, \mu)$. Then, for any $x, y \in \Omega$ we have
\begin{equation*}
(f(x)^2 -f(y)^2)_+^2 = (|f(x)| + |f(y)|)^2(|f(x)| - |f(y)|)_+^2
\le 4 |f(x)|^2 (|f(x)| - |f(y)|)_+^2.
\end{equation*}
Taking integrals and roots, we thus get that for any function $f \colon \Omega \to \mathbb{R}$ in $L^2(\mu)$ such that $|\mathfrak{d}^+ |f||
\le 1$, we have
\begin{equation}
\label{Kettenregelimitation}
|\mathfrak{d}^+ f^2| \le 2 |f|.
\end{equation}
The same holds for product measures, i.\,e. the multivariate case.

In the sequel, we also need modified LSI results for $\mathfrak{d}^+$. It is easily seen
that if some measurable space $(\Omega, \mathcal{A})$ equipped with a probability measure $\mu$ satisfies the modified LSI
\eqref{mLSI} with Sobolev constant $\sigma^2 > 0$ with respect to $\mathfrak{d}$, it also satisfies the modified LSI \eqref{mLSI}
with respect to $\mathfrak{d}^+$, and the Sobolev constant can be chosen $2 \sigma^2$.
Hence, we can transport Propositions \ref{DLSUfueralle} and \ref{DLSUfuerBernoulli} and Lemma
\ref{DLSU-Prod} to the $\mathfrak{d}^+$ difference operators. In fact, results of this type can already be found in \cite{L3} (Proposition 5.8) or \cite{B-L-M} (e.\,g. Proposition 10).

\begin{proposition}
	\label{DLSUnablaplus}
	For all $i = 1, \ldots, n$, let $(\Omega_i, \mathcal{A}_i)$ be measurable spaces equipped with probability measures $\mu_i$. Then,
	the product measure $\mu_1 \otimes \ldots \otimes \mu_n$ on $(\Omega_1 \times \ldots \times \Omega_n, \mathcal{A}_1 \otimes \ldots \otimes
	\mathcal{A}_n)$ satisfies the modified LSI \eqref{mLSI} with Sobolev constant $\sigma^2 = 4$ with respect to $\mathfrak{d}^+$ as in
	\eqref{DiskrDiff3}. If all the $\Omega_i$ are two-point spaces, we can take $\sigma^2 = 2$.
\end{proposition}

Now \eqref{Kettenregelimitation} leads us back to the basic inequality needed to estimate large deviations in \cite{B-G1}. We
therefore arrive at the following lemma:

\begin{lemma}
Let $\mu$ be a probability measure on some measurable space $(\Omega, \mathcal{A})$ which satisfies the modified LSI \eqref{mLSI}
with Sobolev constant $\tilde{\sigma}^2 > 0$ with respect to the gradient operator $\mathfrak{d}^+$ from \eqref{DiskrDiff3}. Moreover,
let $f \colon \Omega \to \mathbb{R}$ be a bounded measurable function such that $|\mathfrak{d} |f|| \le 1$. Then, for all $t \in [0, \frac{1}{2 \tilde{\sigma}^2})$ we have
\begin{equation}
\label{B'}
\int e^{t f^2} d\mu \leq \exp\left(\frac{t}{1 - 2 \tilde{\sigma}^2 t} \int f^2 d\mu\right).
\end{equation}
\label{ExpUnglVorstufe}
\end{lemma}

\begin{proof}
	We adapt the arguments from \cite{B-G1}, p. 6 f. First, consider the inequality
	\begin{equation}
	\int e^f d\mu \le \left(\int e^{\lambda f + (1 - \lambda) \tilde{\sigma}^2 |\mathfrak{d}^+ f|^2/2} d\mu\right)^{1/\lambda}
	\label{B-G1-Ungl}
	\end{equation}
	for all bounded measurable functions $f \colon \Omega \to \mathbb{R}$ and all $\lambda \in (0,1]$. Here we have already plugged
	in $\mathfrak{d}^+$ as our choice of the difference operator.
	To deduce \eqref{B-G1-Ungl}, we use the well-known ``variational formula''
	$$\text{Ent}(g) = \sup\Big\{\int gh \, d\mu \colon h \colon \Omega \to \mathbb{R}\enskip \text{measurable s.\,th.}\enskip
	\int e^h d\mu \le 1 \Big\},$$
	which can be shown by Young's inequality in the form
	$uv \le u \log u - u + e^v$
	for all $u \ge 0$ and $v \in \mathbb{R}$, for instance. See \cite[Proposition 5.6]{L3} for details.
	If we set $g := e^f$ and $h := \lambda f + (1-\lambda) \tilde{\sigma}^2 |\mathfrak{d}^+ f|^2/2 - \beta$ with $\beta = \log \int
	e^{\lambda f + (1 - \lambda) \tilde{\sigma}^2 |\mathfrak{d}^+ f|^2/2} d\mu$, we have $\int e^h d\mu = 1$ and thus
	$$\int (\lambda f + (1 - \lambda) \tilde{\sigma}^2 |\mathfrak{d}^+ f|^2/2 - \beta) e^f d\mu \le \text{Ent}(e^f).$$
	Since $f$ satisfies the modified LSI \eqref{mLSI} with constant $\tilde{\sigma}^2$, it follows that
	\begin{align*}
	&\lambda \int f e^f d\mu + (1-\lambda) \text{Ent}(e^f) - \beta \int e^f d\mu \le \text{Ent}(e^f)\\ \Leftrightarrow\quad
	&\lambda \int e^f d\mu \log \int e^f d\mu - \beta \int e^f d\mu \le 0,
	\end{align*}
	from which we directly get \eqref{B-G1-Ungl}.
	
	We now apply \eqref{B-G1-Ungl} to the function $sf^2/(2 \tilde{\sigma}^2)$ with $0 < s < 1$ and $\lambda = (p-s)/(1-s)$ for
	any $p \in (s,1]$. Together with \eqref{Kettenregelimitation} (note that $|\mathfrak{d} |f|| \le 1$ implies $|\mathfrak{d}^+ |f||
	\le 1$), this gives
	$$\int e^{s f^2/(2 \tilde{\sigma}^2)} d\mu \le \Big(\int \exp\Big(\frac{psf^2}{2 \tilde{\sigma}^2}\Big) d\mu\Big)^{(1-s)/(p-s)}.$$
	For $p=1$ both sides are equal, and as for $p<1$ the upper inequality holds, we get that the logarithm of the left-hand side
	(considered as a function of $p$) must increase more rapidly at $p=1$ than that of the right-hand side. We thus consider
	the derivatives of the logarithms of both sides at $p=1$ and arrive at the inequality
	$$0 \ge \frac{1}{1-s} \bigg[(1-s) \int \frac{sf^2}{2 \tilde{\sigma}^2} e^{sf^2/(2 \tilde{\sigma}^2)} d\mu -
	\int e^{sf^2/(2 \tilde{\sigma}^2)} d\mu \log \int e^{sf^2/(2 \tilde{\sigma}^2)} d\mu \bigg].$$
	
	Now we set
	$u(s) := \int e^{sf^2/(2 \tilde{\sigma}^2)} d\mu$,
	$s \in (0,1]$. Then we get
	$$0 \ge \frac{1}{1-s}\left[s(1-s) u'(s) - u(s) \log u(s)\right]\quad \Leftrightarrow\quad 0 \ge \frac{1-s}{s} \frac{u'(s)}{u(s)} - \frac{1}{s^2} \log u(s).$$
	Hence, the function
	$v(s) := \exp(\frac{1-s}{s} \log u(s))$
	is non-increasing in $s$, and therefore we have $v(s) \le \lim_{s \downarrow 0} v(s) =: v(0^+)$ for all $s \in (0, 1]$.
	
	Note that
	\begin{align*}
	v(0^+) = \lim_{s \downarrow 0} \Big(u(s)^{(1-s)/s}\Big)
	= \lim_{s \downarrow 0} \Big(\int e^{sf^2/(2 \tilde{\sigma}^2)} d\mu\Big)^{(1-s)/s} = \exp \Big(\frac{1}{2 \tilde{\sigma}^2} \int f^2 d\mu\Big).
	\end{align*}
	Thus, for all $s \in (0,1]$ we have
	\begin{align*}
	&\exp\Big(\frac{1-s}{s} \log u(s) \Big) \le \exp\Big(\frac{1}{2 \tilde{\sigma}^2} \int f^2 d\mu\Big)\\
	\Leftrightarrow \quad
	&\int e^{sf^2/(2 \tilde{\sigma}^2)} d\mu \le \exp \Big(\frac{1}{2 \tilde{\sigma}^2}\frac{s}{1-s} \int f^2 d\mu \Big).
	\end{align*}
	Setting $t = s/(2 \tilde{\sigma}^2)$ completes the proof.
\end{proof}

Combining inequalities \eqref{A'} and \eqref{B'}, we now get the following result.

\begin{proposition}
\label{ExponentielleUngl}
Let $\mu$ be a probability measure on some measurable space $(\Omega, \mathcal{A})$ which satisfies the modified LSI \eqref{mLSI} with
Sobolev constant
$\sigma^2 > 0$ with respect to $\mathfrak{d}$ and which moreover satisfies the modified LSI \eqref{mLSI}
with Sobolev constant $\tilde{\sigma}^2$ with respect
to $\mathfrak{d}^+$. Furthermore, let $f \colon \Omega \to \mathbb{R}$ be a bounded measurable function such that $\int
f d\mu = 0$ and $|\mathfrak{d}|\mathfrak{d} f|| \le 1$. Then, we have
\begin{equation}
\label{expUnglFormel}
\int \exp \left(\frac{1}{2 \sigma \tilde{\sigma}} f\right) d\mu \le \exp\left(\frac{1}{2 \tilde{\sigma}^2} \int |\mathfrak{d} f|^2
d\mu \right).
\end{equation}
\end{proposition}

\begin{proof}
First applying \eqref{A'} to $\lambda f$ and then \eqref{B'} with $t = \lambda^2 \sigma^2$ for any $\lambda \in [0,\frac{1}{\sqrt{2} \sigma \tilde{\sigma}})$ and with
$f$ replaced by $|\mathfrak{d} f|$ leads to
$$\int e^{\lambda f} d\mu \leq \int e^{\lambda^2 \sigma^2 |\mathfrak{d} f|^2} d\mu \le \exp \Big(\frac{\lambda^2 \sigma^2}{1 - 2 \sigma^2 \tilde{\sigma}^2 \lambda^2} \int |\mathfrak{d} f|^2
d\mu\Big).$$
%Moreover, \eqref{B'} with $t = \lambda^2 \sigma^2$ for any $\lambda \in [0,\frac{1}{\sqrt{2} \sigma \tilde{\sigma}})$ and with
%$f$ replaced by $|\mathfrak{d} f|$ gives us
%$$\int e^{\lambda^2 \sigma^2 |\mathfrak{d} f|^2} d\mu \le \exp \left(\frac{\lambda^2 \sigma^2}{1 - 2 \sigma^2 \tilde{\sigma}^2 \lambda^2}
%\int |\mathfrak{d} f|^2 d\mu\right).$$
%Combining these two inequalities then yields
%$$\int e^{\lambda f} d\mu \le \exp \left(\frac{\lambda^2 \sigma^2}{1 - 2 \sigma^2 \tilde{\sigma}^2 \lambda^2} \int |\mathfrak{d} f|^2
%d\mu\right).$$
Setting $\lambda = \frac{1}{2 \sigma \tilde{\sigma}}$ completes the proof.
\end{proof}

\section{Relating First and Second Order Difference Operators}

In order to remove the first order difference operator on the right-hand side of \eqref{expUnglFormel},
we may now study relations of the form
$\gamma \int |\mathfrak{d} f|^2 d \mu \le \int \lVert \mathfrak{d}^{(2)} f \rVert_{\text{HS}}^2 d \mu$
for some constant $\gamma > 0$. Note that due to \eqref{L^2Normengleich}, we may replace $\mathfrak{d}^{(2)} f$ by ``Hoeffding'' differences $\mathfrak{D}^{(2)} f$ on the right-hand side, which enables us to make use of the ``harmonic analysis'' arguments established in Section 2.
%where $\mathfrak{D}^{(2)} f$ is the ``de-diagonalized'' Hessian of $f$ with respect to the ``Hoeffding'' type difference operator $\mathfrak{D}$ introduced in \eqref{Hessediskret}.
Indeed, one of our main tools is the following
lemma about partial integration and self-adjointness
for difference operators and the discrete Laplacian $\mathfrak{L}$ defined on functions of independent random variables.

\begin{lemma}
	\label{selbstadj}
	Let $(\Omega_i, \mathcal{A}_i, \mu_i)$ be probability spaces, and denote by $(\Omega, \mathcal{A}, \mu) := \otimes_{i=1}^n
	(\Omega_i,\linebreak[2] \mathcal{A}_i, \mu_i)$ their product. Let $\mathfrak{D} = (\mathfrak{D}_i)_i$ be the difference operator from \eqref{DiskrDiffFormel}, and let $\mathfrak{L}$ be the Laplacian as in \eqref{Laplace}. Then, for any $f, g \in L^2(\mu)$ we have:
	\begin{enumerate}
		\item $$\int (\mathfrak{D}_i f)g d\mu = \int f(\mathfrak{D}_i g) d\mu = \int (\mathfrak{D}_i f)(\mathfrak{D}_i g) d\mu.$$
		\item $$\int (\mathfrak{D}f)g d\mu = \int f(\mathfrak{D}g) d\mu,$$
		where $\mathfrak{D}$ the integral has to be understood componentwise.
		\item $$\int (\mathfrak{L} f) g d\mu = \int f (\mathfrak{L} g) d \mu = \sum_{i \ne j} \int (\mathfrak{D}_{ij}f)(\mathfrak{D}_{ij}g) d\mu.$$
	\end{enumerate}
\end{lemma}

\begin{proof}
	The proof is elementary. Note that in order to prove (2) and (3), we only need to check (1). Part 1 in turn follows from the fact that
	by Fubini's theorem, we have
	$$\int g \Big(\int f d\mu_i \Big) d\mu = \int \Big(\int f d\mu_i\Big) \Big( \int g d\mu_i\Big) d\mu = \int f \Big( \int g d\mu_i\Big) d\mu.$$
	For (3), note that we always have $\mathfrak{D}_{ij} f = \mathfrak{D}_{ji} f$ for any $i, j$ by \eqref{DiskrDiffFormel} and Fubini's theorem.
\end{proof}

Using this result, we can prove an inequality of the desired type:

\begin{proposition}
	\label{Gradient-Hesse}
	Let $(\Omega_i, \mathcal{A}_i, \mu_i)$ be probability spaces, and denote by $(\Omega, \mathcal{A}, \mu) := \otimes_{i=1}^n
	(\Omega_i, \mathcal{A}_i, \mu_i)$ their product.
	Let $f \in L^2(\mu)$ be a function such that its Hoeff\-ding decomposition with respect to $\mu$ is given by
	$f = \sum_{k=d}^{n} f_k$
	for some $d \ge 2$. Then, we have
	$$\int |\mathfrak{d} f|^2 d\mu \le \frac{1}{d-1} \int \lVert \mathfrak{d}^{(2)} f \rVert_{\text{\emph{HS}}}^2 d\mu.$$
	Equality holds if $f = f_d$, i.\,e. the Hoeffding decomposition of $f$ consists of a single term only.
	Here, $\lVert \cdot \rVert_{\text{\emph{HS}}}$ denotes the Hilbert Schmidt norm of a matrix.
\end{proposition}

\begin{proof}
	First, let $f = f_k$. Then, applying Lemma \ref{selbstadj}(3) leads to
	$$\int \lVert \mathfrak{D}^{(2)} f_k \rVert_{\text{HS}}^2 d\mu = \sum_{i \neq j} \int (\mathfrak{D}_{ij} f_k)(\mathfrak{D}_{ij} f_k) d\mu
	= \int f_k \mathfrak{L} f_k d\mu.$$
	Moreover, Theorem \ref{diagonal} yields $\mathfrak{L} f_k = (k)_2 f_k$. Together with \eqref{L^2Normengleich}, this yields
	$$\int \lVert \mathfrak{D}^{(2)} f_k \rVert_{\text{HS}}^2 d\mu = (k)_2 \int f_k^2 d\mu. \qquad (*)$$
	
	On the other hand, if $X_1, \ldots, X_n$ is a
	sequence of independent random variables with distributions $\mu_i$, $i = 1, \ldots, n$, we have $f_k(X_1, \ldots, X_n) =
	\sum_{i_1< \ldots < i_k} h_{i_1 \ldots i_k}(X_{i_1}, \ldots, X_{i_k})$, where the summands on the right-hand side are pairwise orthogonal
	in $L^2$. Here we used the notation of the proof of Theorem \ref{diagonal}.
	
	Now let $\bar{X}_1, \ldots, \bar{X}_n$ be a sequence of independent copies of the random variables $X_1, \ldots, X_n$, and additionally consider
	the functions $T_{i_j}h_{i_1 \ldots i_k}(X_{i_1}, \ldots, X_{i_d}) = h_{i_1 \ldots i_k}(X_{i_1},\linebreak[2] \ldots, \bar{X}_{i_j}, \ldots, X_{i_k})$ (cf.
	Example \ref{DiskrDiff}(3)). Then,
	$$\bigcup_{i_1 < \ldots < i_k} \{h_{i_1 \ldots i_k}(X_{i_1}, \ldots, X_{i_k}) \} \cup \{T_{i_j}h_{i_1 \ldots i_k}(X_{i_1}, \ldots, X_{i_k}), j = 1, \ldots, k \}$$
	is still a (larger) family of pairwise orthogonal functions in $L^2$, now integrating with respect to the $X_i$ and the $\bar{X}_i$.
	
	Similarly to the deduction of \eqref{D_i f_d}, we therefore get
	\begin{align*}
	&(\mathfrak{d}_if_k(X_1, \ldots, X_n))^2 = \frac{1}{2} \bar{\mathbb{E}}_i (f_k - T_if_k)^2\\
	= &\frac{1}{2} \bar{\mathbb{E}}_i \Big(\sum_{\substack{i_1< \ldots < i_k \\ i \in \{i_1, \ldots, i_k\}}}
	(h_{i_1 \ldots i_k}(X_{i_1}, \ldots, X_{i_k}) - T_ih_{i_1 \ldots i_k}(X_{i_1}, \ldots, X_{i_k}))\Big)^2.
	\end{align*}
	Using orthogonality, it follows that
	\begin{align*}
	&\mathbb{E}(\mathfrak{d}_if_k(X_1, \ldots, X_n))^2\\
	= &\sum_{\substack{i_1< \ldots < i_k \\ i \in \{i_1, \ldots, i_k\}}} \frac{1}{2}\Big(\mathbb{E}\bar{\mathbb{E}}_i(h_{i_1 \ldots i_k}^2
	(X_{i_1}, \ldots, X_{i_k}) + T_ih_{i_1 \ldots i_k}^2(X_{i_1}, \ldots, X_{i_k}))\Big)\\
	= &\sum_{\substack{i_1< \ldots < i_k \\ i \in \{i_1, \ldots, i_k\}}} \mathbb{E}h_{i_1 \ldots i_k}^2(X_{i_1}, \ldots, X_{i_k}).
	\end{align*}
	
	As in the proof of Theorem \ref{diagonal}, it remains to check how often each term $\mathbb{E} h_{i_1 \ldots i_k}^2
	(X_{i_1},\linebreak[2] \ldots, X_{i_k})$ appears in $\mathbb{E}|\mathfrak{d} f_k|^2 = \sum_i \mathbb{E} (\mathfrak{d}_i f_k)^2$.
	However, it is clear that
	each $i \in \{i_1, \ldots, i_k\}$ replicates the summand $\mathbb{E} h_{i_1 \ldots i_k}(X_{i_1}, \ldots, X_{i_k})$ exactly once.
	Consequently, it follows that $\mathbb{E}|\mathfrak{d} f_k|^2 = k \mathbb{E}f_k^2$, or
	$$\int |\mathfrak{d} f_k|^2 d\mu = k \int f_k^2 d\mu. \qquad (**)$$
	Comparing $(*)$ and $(**)$ completes the proof in case of $f = f_k$.
	
	For functions with arbitrary Hoeffding expansion we shall use the orthogonality of the terms of the Hoeffding decomposition to get
	$$\int |\mathfrak{d} f|^2 d\mu = \sum_{k=d}^n \frac{1}{k-1} \int \lVert \mathfrak{D}^{(2)} f_k \rVert_{\text{HS}}^2 d\mu
	\le \frac{1}{d-1} \int \lVert \mathfrak{D}^{(2)} f \rVert_{\text{HS}}^2 d\mu.$$
	In view of \eqref{L^2Normengleich}, this finally completes the proof.
\end{proof}

%\section{Proof of the Main Theorems}

We are now ready to prove Theorem \ref{zentral}. In fact, using the results established in Sections 2--4, we may easily obtain some complementary results which can be shown along the lines of the proof of Theorem \ref{zentral}. For instance, we have the following slight sharpening of Theorem \ref{zentral} if all the measures $\mu_i$ are Bernoulli measures:

\begin{proposition}
\label{zentralfuerBernoulli}
Using the notations of Theorem \ref{zentral}, let all the $\mu_i$ be of the form $\mu_i = p_i\delta_{a_i} + (1-p_i)\delta_{b_i}$,
where $a_i, b_i \in \mathbb{R}$, $p_i \in (0,1)$ for all $i$, and $\delta_x$ denotes the Dirac measure at $x \in \mathbb{R}$.
Then, assuming the conditions of Theorem \ref{zentral}, we have
$$\int \exp \left(\frac{1}{3 + 2b^2} |f|\right) d\mu \le 2.$$
\end{proposition}

We now prove give a joint proof of Theorem \ref{zentral} and Proposition \ref{zentralfuerBernoulli}.

\begin{proof}[Proof of Theorem \ref{zentral} and Proposition \ref{zentralfuerBernoulli}] First, combining Proposition
\ref{ExponentielleUngl}, Proposition \ref{Gradient-Hesse} with $d=2$ and the assumptions from Theorem \ref{zentral} leads to
\begin{equation}
\label{BeginnBew}
\int \exp \Big(\frac{1}{2 \sigma \tilde{\sigma}} f \Big) d\mu \le \exp \Big(\frac{1}{2 \tilde{\sigma}^2} \int
\lVert \mathfrak{d}^{(2)} f \rVert_{\text{HS}}^2 d\mu \Big) \le \exp \Big(\frac{b^2}{2 \tilde{\sigma}^2} \Big)
\end{equation}
if $\mu$ satisfies the modified LSI \eqref{mLSI} with constant $\sigma^2 > 0$ with respect to $\mathfrak{d}$ and furthermore
with constant $\tilde{\sigma}^2 > 0$ with respect to $\mathfrak{d}^+$.
Now, from \eqref{BeginnBew} we get
\begin{equation*}
\int \exp \Big(\frac{1}{2 \sigma \tilde{\sigma}} |f|\Big) d\mu
\le \int \Big(\exp \Big(\frac{1}{2 \sigma \tilde{\sigma}}f\Big) + \exp \Big(\frac{1}{2 \sigma \tilde{\sigma}}(-f)\Big)\Big) d\mu
\le 2 \exp \Big(\frac{b^2}{2 \tilde{\sigma}^2} \Big).
\end{equation*}
Thus, by applying H\"{o}lder's inequality we obtain
$$\int \exp \Big(\frac{1}{2 \sigma \tilde{\sigma} \kappa} |f|\Big) d\mu \le
\Big(\int \exp \Big(\frac{1}{2 \sigma \tilde{\sigma}} |f|\Big) d\mu \Big)^{1/\kappa}
\le \Big(2 \exp \Big(\frac{b^2}{2 \tilde{\sigma}^2} \Big)\Big)^{1/\kappa}$$
for all $\kappa \ge 1$. The last term is bounded by $2$ if
$\kappa \ge (\log 2 + b^2/(2 \tilde{\sigma}^2))/\log 2$,
or equivalently
$1/(2 \sigma \tilde{\sigma} \kappa) \le \log 2/(2 \sigma \tilde{\sigma} \log 2 + \sigma \tilde{\sigma}^{-1} b^2)$.

By Proposition \ref{DLSUfueralle}, Proposition \ref{DLSUfuerBernoulli}, Lemma \ref{DLSU-Prod} and Proposition \ref{DLSUnablaplus},
we can set $\sigma^2 = 2$ and $\tilde{\sigma}^2 = 4$ or, in the Bernoulli case, $\sigma^2 = 1$ and $\tilde{\sigma}^2 = 2$. We thus choose
\begin{equation}
\label{R1}
\int \exp \Big(\frac{\log 2}{\sqrt{32} \log 2 + \frac{1}{\sqrt{2}} b^2} |f| \Big) d\mu \le 2,\qquad
\int \exp \Big(\frac{\log 2}{\sqrt{8} \log 2 + \frac{1}{\sqrt{2}} b^2} |f| \Big) d\mu \le 2
\end{equation}
for $\sigma^2 = 2$ and $\tilde{\sigma}^2 = 4$ or $\sigma^2 = 1$ and $\tilde{\sigma}^2 = 2$, respectively. The proof of completed by noting that for all $x \ge 0$,
\begin{equation*}
\frac{\log 2}{\sqrt{32} \log 2 + \frac{1}{\sqrt{2}} x} \ge \frac{1}{6 + 2x}\qquad \text{and}\qquad
\frac{\log 2}{\sqrt{8} \log 2 + \frac{1}{\sqrt{2}} x} \ge \frac{1}{3 + 2x}.
\end{equation*}
\end{proof}

Moreover, it is straightforward to reformulate Theorem \ref{zentral} using $\mathfrak{d}^+$ instead of $\mathfrak{d}$. The proof is easily obtained by simple modifications of the above arguments:

\begin{proposition}
\label{zentral+}
Using the notations of Theorem \ref{zentral}, we require that
$|\mathfrak{d}^+|\mathfrak{d}^+ f|| \le 1$,
where $\mathfrak{d}^+$ is the difference operator from \eqref{DiskrDiff3}.
Then, we have
$$\int \exp \left(\frac{1}{2(4 + b^2)} |f|\right) d\mu \le 2.$$
\end{proposition}

\begin{proof}
Note that we can use \eqref{A'} with $\mathfrak{d}$ replaced by $\mathfrak{d}^+$, that is
$\int e^f d\mu \leq \int e^{\tilde{\sigma}^2 |\mathfrak{d}^+ f|^2} d\mu$
for any bounded measurable function $f \colon \Omega \to \mathbb{R}$ with $\int f d\mu = 0$. Proceeding as in Section 4 then
leads to the inequality
$$\int \exp \Big(\frac{1}{2 \tilde{\sigma}^2} f\Big) d\mu \le \exp\Big(\frac{1}{2 \tilde{\sigma}^2} \int |\mathfrak{d}^+ f|^2
d\mu \Big)$$
if $f \colon \Omega \to \mathbb{R}$ is any bounded measurable function such that $\int f d\mu = 0$ and $|\mathfrak{d}^+|\mathfrak{d}^+ f|| \le 1$.
Since
$\int |\mathfrak{d}^+ f|^2 d\mu \le \int |\mathfrak{d} f|^2 d\mu$,
we can now use Proposition \ref{Gradient-Hesse} as well.
The remaining part of the proof is similar to the proof of Theorem \ref{zentral}. Thus we finally arrive at the inequality
$1/(2 \tilde{\sigma}^2 \kappa) \le \log 2/(2 \tilde{\sigma}^2 \log 2 + b^2)$.
Plugging in $\tilde{\sigma}^2 = 4$ and noting that
$\log 2/(8 \log 2 + x) \ge 1/(8 + 2x)$
for all $x \ge 0$ completes the proof.
\end{proof}

Finally, we prove Theorem \ref{mit Hoeffding1}.

\begin{proof}[Proof of Theorem \ref{mit Hoeffding1}]
	The basic argument is as follows: if we have two functions $\varphi_1$ and $\varphi_2$ on $\mathbb{R}^n$ both satisfying
	$\int e^{c_i |\varphi_i|} d\mu \le 2$
	for some constants $c_i > 0$, $i = 1, 2$, it follows that
	\begin{align}
	\label{Argument}
	\begin{split}
	\int e^{\min(c_1, c_2) |\varphi_1 + \varphi_2|/2} d\mu
	&\le \int e^{c_1 |\varphi_1|/2} e^{c_2 |\varphi_2|/2} d\mu\\
	&\le \Big(\int e^{c_1 |\varphi_1|} d\mu\Big)^{1/2} \Big(\int e^{c_2 |\varphi_2|} d\mu\Big)^{1/2}
	\le 2
	\end{split}
	\end{align}
	due to the Cauchy--Schwarz inequality. In our situation, we set $\varphi_1 = f_1$ and $\varphi_2 = Rf$.%Hence, we only have to check $(*)$.
	
	The bound for $Rf$ is obvious by Theorem \ref{zentral} and the fact that $\mathfrak{D}_{ij} R f = \mathfrak{D}_{ij} f$
	for all $i \neq j$ in view of \eqref{D_ij f_d}. This leads to
	$c_2 = 1/(6 + 2 b^2)$.
	It remains to bound $f_1$. Here, inequality \eqref{A'} yields
	$$\int e^{\lambda f_1} d\mu \leq \int e^{\sigma^2 \lambda^2 |\mathfrak{d} f_1|^2} d\mu \le e^{\sigma^2 \lambda^2 b_0^2}$$
	for any $\lambda > 0$, thus
	$\int e^{\lambda |f_1|} d\mu \le 2e^{\sigma^2 \lambda^2 b_0^2}$.
	As in the proof of Theorem \ref{zentral}, it follows that
	$$\int e^{\lambda |f_1|/\kappa} d\mu \le \Big(2e^{\sigma^2 \lambda^2 b_0^2}\Big)^{1/\kappa}$$
	for all $\kappa \ge 1$. The right-hand side is bounded by $2$ if
	$\lambda/\kappa \le \lambda \log 2/(\log 2 + \lambda^2 \sigma^2 b_0^2)$.
	Here, the expression on the
	right-hand side attains a maximum at $\lambda = (\log 2)^{1/2}/(\sigma b_0)$ whose value is $(\log 2)^{1/2}/(2 \sigma b_0)$.
	Plugging in $\sigma^2 = 2$, we get
	$c_1/2 = (\log 2)^{1/2}/(4 \sqrt{2} b_0) \ge 1/(7 b_0)$,
	and hence we can estimate $\min (c_1, c_2)/2$ as stated in Theorem \ref{mit Hoeffding1}.
\end{proof}

%Compared to Theorem \ref{zentral}, Theorem \ref{mit Hoeffding1} needs conditions which are more involved and hence are not always easy to check.

\section{Evaluating Second Order Difference Operators}

In Theorem \ref{zentral}, checking the condition $\int \lVert \mathfrak{d}^{(2)} f \rVert_{\text{HS}}^2 d\mu \le b^2$ is typically straightforward,
once we know the Hoeffding decomposition of $f$ (cf. \eqref{L^2Normengleich}, enabling us to use the ``Hoeffding'' differences $\mathfrak{D}$).
In contrast, evaluating the condition $|\mathfrak{d}|\mathfrak{d} f|| \le 1$ tends to be more
involved. Therefore, we shall provide a reformulated version of Theorem \ref{zentral}
with conditions which are easier to apply.

\begin{theorem}
\label{einfachereBed}
Let $(\Omega_i, \mathcal{A}_i, \mu_i)$ be probability spaces, and denote by $(\Omega, \mathcal{A}, \mu) := \otimes_{i=1}^n
(\Omega_i,\linebreak[2] \mathcal{A}_i, \mu_i)$ their product.
Moreover, let $f \colon \Omega \to \mathbb{R}$ be a bounded measurable function so that its Hoeffding decomposition
with respect to $\mu$ is given by
$f = \sum_{k=2}^{n} f_k$.
%Denote by $\mathfrak{d}$ the difference operator introduced in \eqref{DiskrDiff2}.
Assume that the conditions
\begin{equation}
\label{Cond1}
\lVert \mathfrak{d}^{(2)} f \rVert_{\text{\emph{HS}}} \le B_1\qquad \text{and}\qquad
\max_{i = 1, \ldots, n}|\mathfrak{d}_i f| \le B_2
\end{equation}
are satisfied for some $B_1, B_2 \ge 0$, where
%$\mathfrak{d}^{(2)} f$ denotes the ``de-diagonalized'' Hessian of $f$ from \eqref{Hessenabla} and
$\lVert \mathfrak{d}^{(2)} f \rVert_{\text{\emph{HS}}}$ denotes the Hilbert--Schmidt norm of $\mathfrak{d}^{(2)} f$.
Then, we have
$$\int \exp \Big(\frac{c}{B_1 + B_2} |f|\Big) d\mu \le 2$$
for some numerical constant $c > 0$. A possible choice is $c = 1/11$. If all the underlying measures $\mu_i$ are two-point
measures, we can take $c = 1/7$.
\end{theorem}

\begin{proof}
For a set of independent random variables $X_1, \ldots, X_n$ with distributions $\mu_i$, write
\begin{equation}
\label{E5}
|\mathfrak{d}|\mathfrak{d} f(X)|| = \Big(\sum_{i=1}^{n}\frac{1}{2} \bar{\mathbb{E}}_i (|\mathfrak{d} f(X)| - |T_i \mathfrak{d} f(X)|)^2\Big)^{1/2}
\end{equation}
with $X = (X_1, \ldots, X_n)$ and $T_k$ as in Remark \ref{DiskrDiff}(3).
Without loss of generality, we may assume that $|\mathfrak{d} f| \ne 0$. To simplify notation, we introduce the convention that
$\sum^{(j)}$ means summation extending over all indexes but $j$. Similarly, $\sum^{(j,k)}$ denotes summation over all
indexes but $j$ and $k$.
Now, setting $a := \sum_{j=1}^{n}{}^{(i)} (\mathfrak{d}_j f)^2$, $b := (\mathfrak{d}_i f)^2$, $c := \sum_{j=1}^{n}{}^{(i)}
(T_i \mathfrak{d}_j f)^2$ and
$d := (T_i \mathfrak{d}_i f)^2$ for any $1 \le i \le n$, we arrive at
\begin{align}
\label{E0}
\begin{split}
(|\mathfrak{d} f| - |T_i \mathfrak{d} f|)^2 &= (\sqrt{a+b} - \sqrt{c+d})^2 = \Big(\frac{a + b - c - d}{\sqrt{a+b} + \sqrt{c+d}}\Big)^2\\ &\le \Big(|\sqrt{a} - \sqrt{c}| + \frac{|b-d|}{\sqrt{a+b}}\Big)^2
\le 2 \Big((\sqrt{a} - \sqrt{c})^2 + \frac{(b-d)^2}{a+b}\Big).
\end{split}
\end{align}
(Using the simpler estimate $|\sqrt{a+b} - \sqrt{c+d}| \le |\sqrt{a} - \sqrt{c}| + |\sqrt{b} - \sqrt{d}|$ instead would essentially lead to a
condition on first order differences only.) Moreover,
\begin{align}
\label{E4}
&(\sqrt{a} - \sqrt{c})^2 = \Big(\big(\sum_{j=1}^{n}{}^{(i)} (\mathfrak{d}_j f)^2\big)^{1/2} - \big(\sum_{j=1}^{n}{}^{(i)}
(T_i \mathfrak{d}_j f)^2\big)^{1/2} \Big)^2\notag\\
\le \; &\sum_{j=1}^{n}{}^{(i)} (\mathfrak{d}_j f - T_i \mathfrak{d}_j f)^2
= \frac{1}{2} \sum_{j=1}^{n}{}^{(i)} \big((\bar{\mathbb{E}}_j (f-T_jf)^2)^{1/2} - (\bar{\mathbb{E}}_j (T_if-T_{ij}f)^2)^{1/2}\big)^2\notag\\
\le \; &\frac{1}{2} \sum_{j=1}^{n}{}^{(i)} \bar{\mathbb{E}}_j (f- T_jf - T_if + T_{ij}f)^2
\end{align}
Combining \eqref{E5}, \eqref{E0} and \eqref{E4} together with the trivial estimate $\sqrt{x+y} \le \sqrt{x} + \sqrt{y}$ for all $x, y \ge 0$
then yields
\begin{equation}
\label{E1}
|\mathfrak{d}|\mathfrak{d} f(X)|| \le \sqrt{2} \Big( \lVert \mathfrak{d}^{(2)} f(X) \rVert_{\text{HS}} + \Big(\frac{1}{2}
\sum_{i=1}^{n} \bar{\mathbb{E}}_{i} \frac{((\mathfrak{d}_i f(X))^2 - (T_i \mathfrak{d}_i f(X))^2)^2}{|\mathfrak{d} f(X)|^2}\Big)^{1/2} \Big).
\end{equation}
We may further estimate the last term by
\begin{equation}
\label{E2}
\Big(\sum_{i=1}^{n} \bar{\mathbb{E}}_{i} \frac{|(\mathfrak{d}_i f(X))^2 - (T_i \mathfrak{d}_i f(X))^2|}{|\mathfrak{d} f(X)|^2}\Big)^{1/2}
\sup_{x \in \text{supp}(\mu)} \max_{i = 1, \ldots, n} |\mathfrak{d}_i f(x)|.
\end{equation}

We now claim that
\begin{equation}
\label{E3}
\Big(\sum_{i=1}^{n} \bar{\mathbb{E}}_{i} \frac{|(\mathfrak{d}_i f(X))^2 - (T_i \mathfrak{d}_i f(X))^2|}{|\mathfrak{d} f(X)|^2}\Big)^{1/2} \le 1.
\end{equation}
To see this, recall that by \eqref{nabla2},
$(\mathfrak{d}_i f(X))^2 = ((\mathfrak{D}_i f(X))^2 + \mathbb{E}_i(\mathfrak{D}_i f(X))^2)/2$,
and therefore
$$|(\mathfrak{d}_i f(X))^2 - (T_i \mathfrak{d}_i f(X))^2| \le ((\mathfrak{D}_i f(X))^2 + (T_i\mathfrak{D}_i f(X))^2)/2.$$
Taking expectations yields
$\bar{\mathbb{E}}_i |(\mathfrak{d}_i f(X))^2 - (T_i \mathfrak{d}_i f(X))^2| \le (\mathfrak{d}_i f(X))^2$,
which proves \eqref{E3}.
%To see this, for simplicity we first assume that the Hoeffding decomposition of $f$ consists of the second order term only, i.\,e. $f(X_1, \ldots, X_n)
%= \sum_{i<j} h_{ij}(X_i, X_j)$. The general case is then proved similarly using the orthogonality of the Hoeffding terms.
%Note that by Remark \ref{DiskrDiff}(3) and the elementary
%properties of the Hoeffding decomposition \eqref{Hoeffding}, we have
%$$(\mathfrak{d}_i f)^2 = \frac{1}{2}\bar{\mathbb{E}}_i(\sum_{j=1}^n{}^{(i)} h_{ij} - \sum_{j=1}^n{}^{(i)}T_ih_{ij})^2
%= \frac{1}{2}\Big((\sum_{j=1}^n{}^{(i)} h_{ij})^2 + \mathbb{E}_i(\sum_{j=1}^n{}^{(i)} h_{ij})^2\Big).$$
%Here we set $h_{ij} = h_{ji}$ for $i > j$. Consequently, we get
%$$|(\mathfrak{d}_i f)^2 - (T_i \mathfrak{d}_i f(X))^2| \le \frac{1}{2} \Big((\sum_{j=1}^n{}^{(i)} h_{ij})^2 +
%(\sum_{j=1}^n{}^{(i)} T_i h_{ij})^2 \Big).$$
%Taking expectations, we obtain
%$$\bar{\mathbb{E}}_i |(\mathfrak{d}_i f)^2 - (T_i \mathfrak{d}_i f(X))^2| \le (\mathfrak{d}_i f)^2,$$
%from which the claim follows in the case $f = f_2$. In the general case, we may replace $\sum_{j=1}^n{}^{(i)} h_{ij}$ by
%$$\sum_{j=1}^n{}^{(i)} h_{ij} + \sum_{j<k}{}^{(i)} h_{ijk} + \ldots,$$
%where the summation extends over all terms up to the order $n$. From here on the proof is similar to the case $f = f_2$.

Combining \eqref{E1}, \eqref{E2} and \eqref{E3} with the assumptions from the theorem, we therefore arrive at
$|\mathfrak{d}|\mathfrak{d} f|| \le \sqrt{2} B_1 + B_2$.
Moreover, by \eqref{nablaiteriert2}, we have
$(\mathfrak{D}_{ij} f(x))^2 \le 4 (\mathfrak{d}_{ij} f(x))^2$ and hence
$$\int \lVert \mathfrak{D}^{(2)} f \rVert_{\text{HS}}^2 d\mu \le 4 B_1^2.\qquad (*)$$
Finally, consider the ``normalized'' function $f/(\sqrt{2}B_1 + B_2)$ and use $(*)$ in \eqref{R1} from the proof of
Theorem \ref{zentral}, respectively. The proof of Theorem \ref{einfachereBed} then follows by elementary computations.
\end{proof}

As for conditions \eqref{Cond1}, note that in typical cases (for instance, if the function $f$ is
symmetric) we have $B_1 = \Theta(B_2)$ as $n \to \infty$.

For functions of independent Rademacher variables taking values in $\{\pm1 \}$, we don't seem to
need first order differences. It is well-known that such functions can be represented in the form
\begin{equation}
\label{Fourier-Walsh}
f(X_1, \ldots, X_n) = \alpha_0 + \sum_{i=1}^{n} \alpha_i X_i + \sum_{i<j} \alpha_{ij} X_iX_j + \ldots,
\end{equation}
where the coefficients $\alpha_I$ (with a suitable multi-index $I$) are real numbers and the summation extends over all terms
up to the order $n$. More precisely, we have
$\alpha_{i_1 \ldots i_d} = \mathbb{E} f(X_1, \ldots, X_n)X_{i_1} \cdots X_{i_d}$
for any $i_1 < \ldots < i_d$, $d = 0, 1, \ldots, n$. This representation is called the \emph{Fourier--Walsh expansion} of the function
$f$, and the expression on the right-hand side of \eqref{Fourier-Walsh} is also known as a \emph{Rademacher chaos}. It is immediately
clear that \eqref{Fourier-Walsh} is at the same time the Hoeffding decomposition of $f$. Applying Corollary \ref{zentralfuerBernoulli}
to functions of this type leads to the following result:

\begin{proposition}
	\label{Bernoullibeispiel2}
	Let $\mu$ be the product measure of $n$ symmetric Bernoulli distributions $\mu_i = \frac{1}{2} \delta_{+1} + \frac{1}{2} \delta_{-1}$
	on $\{\pm1 \}$, and define $f \colon \mathbb{R}^n \to \mathbb{R}$ by
	$f(x_1, \ldots, x_n) := \sum_{i<j} \alpha_{ij} x_ix_j + \sum_{i<j<k} \alpha_{ijk} x_ix_jx_k + \ldots$,
	where the sum goes up to order $n$ and the $\alpha_{i_1 \ldots i_d}$ are any real numbers. Set
	$B := \sup_{x \in \{\pm1 \}^n} \lVert \mathfrak{D}^{(2)} f(x)\rVert_{\text{\emph{HS}}}$
	with $\mathfrak{D}^{(2)} f(x)$ as in \eqref{Hessediskret}. Then, we have
	$$\int \exp \left(\frac{1}{5 B} |f| \right) d\mu \le 2.$$
\end{proposition}

\begin{proof}
First note that similarly to Remark \ref{DiskrDiff}(1), for products of symmetric Bernoulli distributions we have $\mathfrak{d}_{ij} f
= |\mathfrak{D}_{ij} f|$ for any $i \ne j$ and consequently $\lVert \mathfrak{d}^{(2)} f\rVert_{\mathrm{HS}} = \lVert \mathfrak{D}^{(2)}
f\rVert_{\mathrm{HS}}$. Therefore, in view of Corollary \ref{zentralfuerBernoulli}, it suffices to prove that
$|\mathfrak{d}|\mathfrak{d}f|| \le \lVert \mathfrak{d}^{(2)} f\rVert_{\mathrm{HS}}$
on $\text{supp}(\mu)$.

To this end, note that for any $i = 1, \ldots, n$, by the fact that $T_i |\mathfrak{d} f| = |T_i \mathfrak{d} f|$ and the reverse triangular inequality,
\begin{equation}\label{Schr1}
(\mathfrak{d}_i|\mathfrak{d} f|)^2 = \frac{1}{2} \bar{\mathbb{E}}_i (|\mathfrak{d} f| - |T_i \mathfrak{d} f|)^2
\le \frac{1}{2} \bar{\mathbb{E}}_i |\mathfrak{d} f - T_i \mathfrak{d} f|^2.
\end{equation}
Here, the difference $\mathfrak{d} f - T_i \mathfrak{d} f$ is defined componentwise. Using the Fourier--Walsh expansion \eqref{Fourier-Walsh} and
the fact that $x_i^2 = 1$ on $\text{supp}(\mu)$, it is easy to see that $T_i \mathfrak{d}_i f = \mathfrak{d}_i f$. Therefore, using the notations from
the proof of Theorem \ref{einfachereBed},
\begin{align}
|\mathfrak{d} f - T_i \mathfrak{d} f|^2
&= \frac{1}{2} \sum_{j=1}^{n}{}^{(i)} \big((\bar{\mathbb{E}}_j (f - T_jf)^2)^{1/2} - (\bar{\mathbb{E}}_j (T_if - T_{ij}f)^2)^{1/2}\big)^2\notag\\
&\le \frac{1}{2} \sum_{j=1}^{n}{}^{(i)} \bar{\mathbb{E}}_j (f - T_jf - T_if + T_{ij}f)^2.\label{Schr2}
\end{align}
Here, the last step follows from the reverse triangular inequality again (for the norm $(\bar{\mathbb{E}}_j(\cdot)^2)^{1/2}$). Combining \eqref{Schr1}
and \eqref{Schr2} and summing over $i = 1, \ldots, n$ finishes the proof.
\end{proof}

\section{Differentiable Functions: Proofs}

In order to prove Theorem \ref{kontinuierlich}, we need to adapt some of the elements of the proof of Theorem \ref{zentral} from
the previous sections. For that, if $(M, d)$ is a metric space and $f \colon M \to \mathbb{R}$ is a continuous function, we may define
the generalized modulus of the gradient by
\begin{equation}
\label{generalizedmodulus}
|\nabla^* f(x)| = \limsup_{y \to x} \frac{|f(x) - f(y)|}{d(x,y)}
\end{equation}
for any $x \in M$, where the limsup is assigned to be zero at isolated points. By the continuity of $f$, $x \mapsto |\nabla^* f(x)|$
is a Borel-measurable function. If $f$ is a differentiable function on some open subset $G \subset \mathbb{R}^n$,
the generalized modulus of the gradient agrees with the Euclidean norm of the usual gradient. We may iterate the generalized modulus
of the gradient by setting for any $x \in M$
\begin{equation}
\label{generalizedmodulusit}
|\nabla^*|\nabla^* f(x)|| := \limsup_{y \to x} \frac{||\nabla^* f(x)| - |\nabla^* f(y)||}{d(x,y)}.
\end{equation}

Using the generalized modulus of the gradient, we have the following analogues of inequalities \eqref{A'} and \eqref{B'} from Section 4.
Let $(M, d)$ be a metric space, equipped with some Borel probability measure $\mu$ which satisfies a logarithmic Sobolev
inequality with constant $\sigma^2$.
Moreover, let $u \colon M \to \mathbb{R}$ be a $\mu$-integrable locally Lipschitz function. Then, we have
\begin{equation}
\int e^{u - \int u d\mu} d\mu \leq \int e^{\sigma^2 |\nabla^* u|^2} d\mu.
\label{A}
\end{equation}
Moreover, if we additionally require
$|\nabla^* u| \leq 1$, we have
\begin{equation}
\int e^{t u^2} d\mu \leq \exp\Big(\frac{t}{1 - 2 \sigma^2 t} \int u^2 d\mu\Big)
\label{B}
\end{equation}
for any $0 \leq t < \frac{1}{2 \sigma^2}$. As mentioned in Section 4, \eqref{A} and \eqref{B} are due to \cite{B-G1} and \cite{A-M-S}.

Now consider $M = G$, where $G \subset \mathbb{R}^n$ is some open subset equipped with the Euclidean metric. By proceeding as in the
proof of Proposition \ref{ExponentielleUngl}, we arrive at the following exponential moment inequality:

\begin{proposition}
\label{ExponentielleUngl2}
Let $G \subset \mathbb{R}^n$ be some open set, and let $\mu$ be a probability measure on $(G, \mathcal{B}(G))$ which satisfies the
logarithmic Sobolev inequality \eqref{LSI} with Sobolev constant $\sigma^2 > 0$. Furthermore, let $f \colon G \to \mathbb{R}$ be a
locally Lipschitz $\mu$-integrable function with $\mu$-mean zero such that $|\nabla^* f|$ is locally Lipschitz and $|\nabla^*|\nabla^* f||
\le 1$. Here, $|\nabla^* f|$ is the generalized modulus of the gradient from \eqref{generalizedmodulus}. Then, we have
$$\int_G \exp \left(\frac{1}{2 \sigma^2} f\right) d\mu \le \exp\left(\frac{1}{2 \sigma^2} \int_G |\nabla^* f|^2 d\mu \right).$$
\end{proposition}

Proposition \ref{ExponentielleUngl2} is a special case of \cite[Proposition 2.1]{B-C-G}. If $f$ is a
$\mathcal{C}^2$-function, the condition $|\nabla^*|\nabla^* f|| \le 1$ can be simplified by the following lemma:

\begin{lemma}
\label{itGradHess}
Let $G \subset \mathbb{R}^n$ be some open set. Then, for any $\mathcal{C}^2$-smooth function $f \colon G \to \mathbb{R}$, the
function $|\nabla^* f|$ is locally Lipschitz and satisfies
$$|\nabla^*|\nabla^* f(x)|| \le \lVert f''(x) \rVert_\text{\emph{Op}}$$
at all points $x \in G$, where $f''(x)$ denotes the Hessian of $f$ at $x \in G$.
\end{lemma}

\begin{proof}
By chain rule, $|\nabla f(x)|$ is differentiable on $\{|\nabla f(x)| \ne 0 \}$ with
$\nabla |\nabla f(x)| = \frac{1}{|\nabla f(x)|} f''(x) \nabla f(x)$,
which immediately yields the desired result if $|\nabla f(x)| \ne 0$.

It remains to consider the case $|\nabla f(x)| = 0$. Here, for any $v \in \mathbb{R}^n$ such that $|v| = 1$, by Taylor expansion
we obtain
$\langle \nabla f(x+h), v \rangle = \langle f''(x) v, h \rangle + o(|h|)$
as $h \to 0$. Here, the $o$-term can be bounded by a quantity which does not depend on the choice of $v$. Therefore, dividing by $|h|$ and
taking limits according to \eqref{generalizedmodulus}, the proof is finished by noting that
\begin{equation*}
|\nabla^* |\nabla f(x)|| = \; \limsup_{h \to 0} \frac{|\nabla f(x+h)|}{|h|} \le \; \sup \Big\{\big\langle f''(x) v, \frac{h}{|h|} \big\rangle \colon |v| = 1, h \ne 0 \Big\} = \lVert f''(x) \rVert_\text{Op}.
\end{equation*}
\end{proof}

%For the function $v(x) = |\nabla^* f(x)|^2 = |\nabla f(x)|^2 = \sum_{j=1}^{n} (\partial_j f(x))^2$, we have
%$$\partial_i v(x) = 2 \sum_{j=1}^n \partial_{ij} f(x) \partial_j f(x),\qquad i = 1, \ldots, n.$$
%Hence, by Cauchy's inequality,
%$$(\partial_i v(x))^2 \le 4 \sum_{j=1}^{n} (\partial_{ij} f(x))^2 \sum_{j=1}^{n} (\partial_j f(x))^2 = 4 |\nabla f(x)|^2 \sum_{j=1}^n (\partial_{ij} f(x))^2,$$
%and thus
%$$|\nabla v(x)| \le 2 |\nabla f(x)| \lVert f''(x)\rVert_\text{HS}.$$
%
%Now consider the function $w(x) = |\nabla f(x)| = \sqrt{v(x)}$. In the region $|\nabla f(x)| \ne 0$, $w$ is differentiable, and we get
%$$|\nabla w(x)| = \frac{|\nabla v(x)|}{2 \sqrt{v(x)}} \le \lVert f''(x) \rVert_\text{HS}$$
%as required.
%
%If $|\nabla f(x)| = 0$, we have $\partial_i f(x) = 0$ for all $i = 1, \ldots, n$. By Taylor expansion at the point $x$, we have
%$$\partial_i f(x+h) = \sum_{j=1}^{n} \partial_{ij} f(x) h_j + o(|h|)\qquad \text{as}\enskip h = (h_1, \ldots, h_n) \to 0,$$
%where $|h| = \sqrt{h_1^2 + \ldots + h_n^2}$. Using Cauchy's inequality, this yields
%$$(\partial_i f(x+h))^2 \le |h|^2 \sum_{j=1}^{n} (\partial_{ij} f(x))^2 + o(|h|^2),$$
%and therefore
%$$|\nabla f(x+h)| \le |h| \lVert f''(x) \rVert_\text{HS} + o(|h|).$$
%Hence, by the definition \eqref{generalizedmodulus},
%$$|\nabla^* w(x)| = \limsup_{h \to 0} \frac{|w(x+h) - w(x)|}{|h|} = \limsup_{h \to 0} \frac{|\nabla f(x+h)|}{|h|}
%\le \lVert f''(x) \rVert_\text{HS}.$$
%Lemma \ref{itGradHess} is proved.

We can now prove Theorems \ref{kontinuierlich} and \ref{kontinuierlich+1.Ordn}:

\begin{proof}[Proof of Theorem \ref{kontinuierlich}]
Given a function $f$ as in Theorem \ref{kontinuierlich}, applying Proposition \ref{ExponentielleUngl2} together with Lemma
\ref{itGradHess} yields
\begin{equation}
\label{expUnglFormel2.5}
\int_G \exp \Big(\frac{1}{2 \sigma^2} f\Big) d\mu \le \exp\Big(\frac{1}{2 \sigma^2} \int_G |\nabla f|^2 d\mu \Big).
\end{equation}
Since $\mu$ satisfies a logarithmic Sobolev inequality with constant $\sigma^2$, it also satisfies a Poincar\'{e} inequality \eqref{PI}
with constant $\sigma^2$. Therefore, since $\int_G \partial_i f d\mu = 0$ for all $i$, we have
$\int_G (\partial_i f)^2 d\mu \le \sigma^2 \sum_{j=1}^{n} \int_G (\partial_{ij} f)^2 d\mu$
for all $i = 1, \ldots, n$, where $\partial_{ij} f(x) = \frac{d^2f(x)}{dx_i dx_j}$. Summing up over all $i$, we get
\begin{equation}
\label{nablaHesse}
\int_G |\nabla f|^2 d\mu \le \sigma^2 \int_G \lVert f'' \rVert_\text{HS}^2 d\mu.
\end{equation}
Combining \eqref{expUnglFormel2.5}, \eqref{nablaHesse} and the assumptions from Theorem \ref{kontinuierlich}, we
arrive at
$$\int_G \exp \Big(\frac{1}{2 \sigma^2} f\Big) d\mu \le \exp\Big(\frac{1}{2} \int_G \lVert f'' \rVert_\text{HS}^2 d\mu \Big)
\le \exp\Big(\frac{b^2}{2}\Big).$$
The rest of the proof is similar to the proof of Theorem \ref{zentral}. We finally arrive at the inequality
$1/(2 \sigma^2 \kappa) \le \log 2/(2 \sigma^2 \log 2 + b^2 \sigma^2)$.
Noting that
$\log 2/(2 \sigma^2 \log 2 + x \sigma^2) \ge 1/(2\sigma^2(1 + x))$
for all $x \ge 0$ finishes the proof.
\end{proof}

\begin{proof}[Proof of Theorem \ref{kontinuierlich+1.Ordn}]
	The proof is similar to the proof of Theorem \ref{mit Hoeffding1} assuming condition $(i)$ from the latter theorem. Setting $\mu[h]
	= \int_G h d\mu$ for any $h \in L^1(\mu)$, write $f = \varphi_1 + \varphi_2$ with
	$\varphi_1(x) = \sum_{i=1}^{n} \mu[\partial_i f] (x_i - \mu[x_i])$, $\varphi_2(x) = f(x) - \varphi_1(x)$.
	We now apply the basic argument \eqref{Argument} from the proof of Theorem \ref{mit Hoeffding1}. Here we need to check that
	$\int e^{c_i |\varphi_i|} d\mu \le 2$
	for $i = 1, 2$ and some constants $c_1, c_2 > 0$. By Theorem \ref{kontinuierlich} applied to $\varphi_2$, we may choose
	$c_2/2 = 1/(4 \sigma^2(1 + b^2))$.
	
	For estimating the function $\varphi_1$, note that $|\nabla \varphi_1|^2 = \sum_{i=1}^{n} (\mu[\partial_i f])^2 \le \sigma^2b_0^2$ by
	assumption. Therefore, applying \eqref{A} yields
	$$\int e^{\lambda \varphi_1} d\mu \leq \int e^{\sigma^2 \lambda^2 |\nabla \varphi_1|^2} d\mu \le e^{\sigma^4 \lambda^2 b_0^2}$$
	for all $\lambda > 0$. Proceeding as in the proof of Theorem \ref{mit Hoeffding1}, we obtain
	$c_1/2 = \linebreak[4] (\log 2)^{1/2}/(4 \sigma^2 b_0) \ge 1/(5 \sigma^2 b_0)$,
	which easily yields the desired result.
\end{proof}

\section{Applications}

\subsection{Functions of Independent Random Variables}

As a first example, we shall consider a certain type of statistics related to Hoeffding-type expansions. In detail, for any $n \in \mathbb{N}$, let $X_1, \ldots, X_n$ be independent random variables and $T_n$ a statistic of the form
\begin{align}
\label{Statistik}
\begin{split}
T_n(X_1, \ldots, X_n) = \; &h_{0,n} + \sum_i h_{1,n} (X_i) n^{-1} + \sum_{i<j} h_{2,n} (X_i, X_j) n^{-2}\\
&+ \sum_{i<j<k} h_{3,n} (X_i, X_j, X_k) n^{-3} + \ldots.
\end{split}
\end{align}
Here, $h_{d,n}$, $d=0, 1, \ldots, n$, are some ``kernel'' functions which are completely degenerate with respect to the $X_i$. Usually, we then have
concentration inequalities of the form
$P(\sqrt{n} (T_n - h_{0,n}) \ge t) \le e^{-ct^2}$,
where $c$ is some absolute constant. Using second order concentration, it is possible to sharpen these bounds. Here we mainly use the results from Section 6.

\begin{example}
\label{HoeffdingEx}
Let $X_1, \ldots, X_n$ be some independent random variables, and let $T_n$ be a statistic of the form \eqref{Statistik}.
Assume we have
\begin{equation}
\label{Bed}
\lVert n \mathfrak{d}^{(2)} T_n \rVert_{\text{\emph{HS}}} \le M\qquad \text{and}\qquad
|n \mathfrak{d}_i (T_n - \sum_i h_{1,n} (X_i) n^{-1})| \le M \; \; \forall i
\end{equation}
for some universal constant $M$ and with $\mathfrak{d}^{(2)} T_n$ as in \eqref{Hessenabla}. Then, there exists some numerical constant $c > 0$ such that
$$P \Big(n |T_n - h_{0,n} - \sum_i h_{1,n} (X_i) n^{-1}| \ge t\Big) \le 2 e^{-ct/M}.$$
\end{example}

This follows immediately from Theorem \ref{einfachereBed}. In particular, conditions \eqref{Bed} are satisfied if $\lVert h_{d,n} \rVert_\infty \equiv \sup_x |h_{d,n} (x)| \le L$ for $d \le m$ and $h_{d,n} \equiv 0$ for all $d \ge m$,
where $m \in \mathbb{N}$ is independent of $n$ and where $L$ is some absolute constant.

A special case is given by functions on the discrete cube, i.\,e. we assume $X_1, \ldots, X_n$ to be i.i.d. random variables with distributions $\mu_i = \frac{1}{2} \delta_{+1} + \frac{1}{2} \delta_{-1}$. In this situation, by Proposition \ref{Bernoullibeispiel2}, we may replace conditions \eqref{Bed} by the single condition
$
\lVert n \mathfrak{D}^{(2)} T_n \rVert_{\text{HS}} \le M
$.
Here, $\mathfrak{D}^{(2)} T_n$ is the ``Hessian'' of $T_n$ with respect to $\mathfrak{D}$ defined in \eqref{Hessediskret}. For instance, if
$T_n(X_1, \ldots, X_n) = \alpha_0 + \sum_i n^{-1} \alpha_i X_i + \sum_{i<j} n^{-2} \alpha_{ij} X_iX_j$ for real numbers $\alpha_0,
\alpha_i, \alpha_{ij}$, then $\lVert n \mathfrak{D}^{(2)} T_n \rVert_{\text{HS}} \le M$ just means $n^{-1} (2 \sum_{i<j} \alpha_{ij}^2)^{1/2} \le M$.

As a second example, we shall consider additive functionals of partial sums, i.\,e. functionals of the form
\begin{equation}
\label{partsum}
S_f := S_f(X) := \sum_{i=1}^{n} f\Big(\sum_{j=1}^{i} X_j\Big).
\end{equation}
Random variables of this kind appear e.\,g. as additive functionals of random walks (cf. \cite{B-I}). Here
we obtain the following result:

\begin{example}
Let $X_1, \ldots, X_n$ be a set of independent random variables, and let $f \colon \mathbb{R} \to \mathbb{R}$ be a bounded measurable function.
Consider $S_f = S_f(X)$ as defined in \eqref{partsum}. Then, there exists some numerical constant $c> 0$ such that for any $t \ge 0$,
$$P(|S_f - \mathbb{E}S_f| \ge t) \le 2 \exp \Big(-c \frac{t}{n^2 \lVert f \rVert_\infty}\Big).$$
\end{example}

\begin{proof}
The proof is obtained by combining Theorem \ref{mit Hoeffding1} and Theorem
\ref{einfachereBed}. For that, we simply have to calculate the respective
differences of first and second order.

To start, note that the first order Hoeffding term of $S_f(X)$ is given by
\begin{equation*}
S^1_f(X) = \sum_{\nu = 1}^{n} \Big(\sum_{i \ge \nu} \big(\mathbb{E}^{(\nu)} f(\sum_{j=1}^{i}X_j) - \mathbb{E} f(\sum_{j=1}^{i}X_j)\big) \Big),
\end{equation*}
where $\mathbb{E}^{(\nu)}$ denotes taking the expectation with respect to all
the random variables $X_1, \ldots, X_n$ but $X_\nu$.
It follows that for any $\nu = 1, \ldots, n$,
\begin{align*}
(\mathfrak{d}_\nu S^1_f(X))^2 = \frac{1}{2} \bar{\mathbb{E}}_\nu
\Big( \sum_{i \ge \nu} \big(\mathbb{E}^{(\nu)} f(\sum_{j=1}^{i} X_j)
- \mathbb{E}^{(\nu)} f(\sum_{j=1}^{i}T_\nu X_j) \big)\Big)^2 \le 2 \lVert f \rVert_\infty^2(n-\nu+1)^2,
\end{align*}
and consequently
$|\mathfrak{d} S^1_f(X)|^2 \le 2 \lVert f \rVert_\infty^2 \sum_{\nu=1}^{n}
(n-\nu+1)^2 = \frac{1}{3} n(n+1)(2n+1) \lVert f \rVert_\infty^2$.

Next we need to check the second order conditions from Theorem \ref{mit Hoeffding1}, i.\,e. \eqref{bedingungen} for $RS^1_f(X) := S_f(X) - S^1_f(X) - \mathbb{E} S_f(X)$ (noting that $\mathfrak{d}^{(2)}RS^1_f(X) = \mathfrak{d}^{(2)}S^1_f(X)$). As shown in (the proof of) Theorem \ref{einfachereBed}, these conditions can be replaced by \eqref{Cond1}.
To see the first condition in \eqref{Cond1}, for any $\nu \ne \mu$,
\begin{align*}
&(\mathfrak{d}_{\nu \mu} RS_f(X))^2 = (\mathfrak{d}_{\nu \mu} S_f(X))^2\\
= \ &\frac{1}{4} \bar{\mathbb{E}}_{\nu \mu} \Big(\sum_{i \ge \nu \vee \mu}\big(f(\sum_{j=1}^{i} X_j) - T_\nu f(\sum_{j=1}^{i} X_j)-T_\mu f(\sum_{j=1}^{i} X_j)+T_{\nu \mu} f(\sum_{j=1}^{i} X_j)\big)\Big)^2\\
\le \ &4 \lVert f \rVert_\infty^2(n-(\nu \vee \mu)+1)^2
\end{align*}
using similar arguments as above, and therefore
$\lVert\mathfrak{d}^{(2)} S_f(X)\rVert_\mathrm{HS}^2 = \sum_{\nu \ne \mu} (\mathfrak{d}_{\nu \mu} S_f(X))^2 \le C n^4 \lVert f \rVert_\infty^2$
for some numerical constant $C > 0$.
Moreover, to see the second condition in \eqref{Cond1}, for any $\nu = 1, \ldots, n$,
\begin{align*}
&\big(\mathfrak{d}_\nu (S_f(X) - S^1_f(X) - \mathbb{E} S_f(X))\big)^2\\
= \ &\frac{1}{2} \bar{\mathbb{E}}_\nu \Big( \sum_{i \ge \nu} \big(f(\sum_{j=1}^{i} X_j) - T_\nu f(\sum_{j=1}^{i} X_j) - \mathbb{E}^{(\nu)} f(\sum_{j=1}^{i}X_j)
+ \mathbb{E}^{(\nu)} T_\nu f(\sum_{j=1}^{i} X_j) \big)\Big)^2\\
\le \ &8 \lVert f \rVert_\infty^2 (n-\nu+1)^2.
\end{align*}
Combining these estimates we easily arrive at the result.
\end{proof}

We may furthermore apply our results in the context of bootstrap methods. Suppose $X_1, \ldots, X_n, \ldots$ are random elements taking values
in $\mathbb{R}^p$ (or some other separable metric space) which are independent and identically distributed from some distribution
$P \in \mathcal{P}_0$. Here, $\mathcal{P}_0$ is a set of probability measures on $\mathbb{R}^p$ which contains all discrete measures.
By $\hat{P}_n$ we denote the empirical measure of the first $n$ observations. Let $T_n \equiv T_n(X_1, \ldots, X_n; P) \equiv
T_n(\hat{P}_n; P)$ be a sequence of symmetric statistics which may depend on the distribution $P$, and let $h$ be a bounded real
function defined on the range of $T_n$.

Here we are interested in estimating $\theta_n(P) := \mathbb{E}_P h(T_n(X_1, \ldots, X_n; P))$. Given $X_1,\linebreak[2] \ldots, X_n$,
Efron's (nonparametric) bootstrap suggests to estimate $\theta_n(P)$ by $\theta_n(\hat{P}_n)$. That is, if we set
$$B_n (P) = \frac{1}{n^n} \sum_{i_1, \ldots, i_n = 1}^{n} h(T_n(X_{i_1}, \ldots, X_{i_n}; P)),$$
Efron's bootstrap is given by $B_n(\hat{P}_n)$. In many situations, this bootstrap can be successfully applied, but in a number of
examples (in particular due to bias problems) it fails asymptotically. These problems have been addressed by D.\,N. Politis and J.\,P. Romano \cite{P-R}, F. G\"{o}tze \cite{G} and P.\,J. Bickel, F. G\"{o}tze and W.\,R. van Zwet \cite{B-G-Z} by introducing the $m$ out of $n$ bootstraps, i.\,e.
sampling from an i.i.d. sample of size $n$ $m$-times independently with or without replacement. For instance, in the case of sampling
without replacement (also called the $\binom{n}{m}$ bootstrap), we consider
$$J_{m,n}(P) = J_m(P) = \frac{1}{\binom{n}{m}} \sum_{i_1 < \ldots < i_m} h(T_m(X_{i_1}, \ldots, X_{i_m}; P))$$
as an estimator of $\theta_n(P)$.
Then, the $\binom{n}{m}$ bootstrap estimator of $J_{m,n}(P)$ is given by $J_{m,n}(\hat{P}_n)$.

In order to access the accuracy of this estimate, one would have to estimate the error involved in replacing $P$ by $\hat{P}_n$ in the functional $J_{m,n}(P)$. Under sufficient smoothness conditions for the dependence on $P$, this would lead to first or second order Hoeffding expansions involving a kernel of $m+1$ or $m+2$ variables, respectively. This would be necessary for evaluating the bias term of this bootstrap estimator. For the sake of brevity, we shall consider the estimate of the variance term for the original $P$ only at this point.

At first order, the variance part of the error has been estimated in \cite{P-R} and \cite{G}. For instance, by \cite[Theorem 1]{B-G-Z}, if $\frac{m}{n} \to 0$, $m \to \infty$, we have
$
J_m(P) = \theta_m(P) + \mathcal{O}_P ((m/n)^{1/2}).
$
Knowing (or at least estimating) the first order Hoeffding term of $J_m(P)$, we may sharpen this result by the
second order results in this paper:

\begin{proposition}
\label{Bootstrap2}
Suppose $\frac{m}{n} \to 0$, $m \to \infty$. Let $h = \sum_{i=0}^{m} h_i$ be the Hoeffding decomposition of $h = h(T_m(X_1, \ldots, X_m; P))$,
and assume that
\begin{equation}
\label{Cond}
|\mathfrak{d}_{ij} h| \le \frac{c_1}{m},\qquad |\mathfrak{d}_i(h-h_0-h_1)| \le c_2
\end{equation}
for all $1 \le i <j \le m$ and all $i = 1, \ldots, m$, respectively, where $c_1$ and $c_2$ are some absolute constants.
Let $J_{m,1}(P)$ denote the first order Hoeffding term of $J_m(P)$.
Then, we have
$$J_m(P) = \theta_m(P) + J_{m,1}(P) + \mathcal{O}_P (\frac{m}{n}).$$
\end{proposition}

\begin{proof}
%We proceed similarly to the proof of \cite[Theorem 1]{B-G-Z}.
%That is, suppose $T_m$ does not depend on $P$. Then, $J_m$ is a $U$-statistic with kernel $h(T_m (x_1, \ldots, x_m))$ and $\mathbb{E}_P J_m = \theta_m(P)$. Set $ R J_m := J_m - \theta_m(P) - J_{m,1}$ (cf. \eqref{Projektion}).
Noting that $\mathbb{E}_P J_m(P) = \theta_m(P)$, let us check the conditions from Theorem \ref{einfachereBed} for  $R J_m (P) := J_m(P) - \theta_m(P) - J_{m,1}(P)$ (cf. \eqref{Projektion}). Using $\mathfrak{d}^{(2)} RJ_m(P) = \mathfrak{d}^{(2)} J_m(P)$ and \eqref{Cond}, by elementary counting arguments we obtain
\begin{align*}
\lVert \mathfrak{d}^{(2)} J_m(P) \rVert_{\text{HS}} &= \Big(\sum_{j \ne k \le n}\Big(\mathfrak{d}_{jk}\frac{1}{\binom{n}{m}} \sum_{\substack{i_1 < \ldots < i_m \\ j,k \in \{i_1, \ldots, i_m\}}} h(X_{i_1}, \ldots, X_{i_m}; P)\Big)^2\Big)^{1/2}\\
&\le \Big(\sum_{j \ne k \le n}\frac{\binom{n-2}{m-2}}{\binom{n}{m}^2} \sum_{\substack{i_1 < \ldots < i_m \\ j,k \in \{i_1, \ldots, i_m\}}} (\mathfrak{d}_{jk}h(X_{i_1}, \ldots, X_{i_m}; P))^2\Big)^{1/2} \le c_1 \frac{m}{n}.
\end{align*}
%Noting that $\mathbb{E}_P J_m(P) = \theta_m(P)$, set $RJ_m(P) := J_m(P) - \theta_m(P) - J_{m,1}(P)$ (cf. \eqref{Projektion}).
%Using the notations of Theorem \ref{einfachereBed} and conditions \eqref{Cond}, we then have
%$$\lVert \mathfrak{d}^{(2)} R J_m(P) \rVert_{\text{HS}} = \Big(\sum_{j \ne k \le n}\Big(\mathfrak{d}_{jk}\frac{1}{\binom{n}{m}} \sum_{\substack{i_1 < \ldots < i_m \\ j,k \in \{i_1, \ldots, i_m\}}} h(X_{i_1}, \ldots, X_{i_m}; P)\Big)^2\Big)^{1/2} \le c_1 \frac{m}{n}$$
Here, to see the first inequality, we may rewrite $\mathfrak{d}_{jk}$ by \eqref{nabla3} and use the inequality $\mathbb{E}(\sum_{i=1}^\nu f_i)^2 \le \nu \sum_{i=1}^\nu \mathbb{E}f_i^2$ with $\mathbb{E}$ replaced by $\bar{\mathbb{E}}_{jk}$ and $\nu = \binom{n-2}{m-2}$.
Similarly, we have
$|\mathfrak{d}_i R J_m(P)| \le c_2 \frac{m}{n}$ for all $i$.
The proof now follows by applying Theorem \ref{einfachereBed}.
\end{proof}

As for the first order Hoeffding term $J_{m,1}(P)$, we have
$J_{m,1}(P) = \sum_{i=1}^n g_1(X_i)$ with
$$g_1(X_i) = \frac{m}{n} \big(\mathbb{E}_P (h(T_m(X_i, X_{j_1}, \ldots, X_{j_{m-1}})) | X_i)
- \mathbb{E}_Ph(T_m(X_i, X_{j_1}, \ldots, X_{j_{m-1}}))\big),$$
where $j_1 < \ldots < j_{m-1}$ is any $(m-1)$-tuple from $\{1, \ldots, n\} \setminus \{i\}$.
Conditions \eqref{Cond} imply that $h(T_m(X_1, \ldots, X_m);P)$ is ``normalized'', i.\,e. we have $B_1 = B_2 = \mathcal{O}(1)$ in Theorem
\ref{einfachereBed} for $f = h - h_0 - h_1$. This may be achieved by requiring $h$ to be sufficiently smooth.

In fact, in many applications, we can only assume $|\mathfrak{d}_{ij} h| \le c_1$. In this case, we still get $J_m(P) = \theta_m(P) + J_{m,1}(P) + \mathcal{O}_P (m^2/n)$ in Proposition \ref{Bootstrap2}. A typical situation is $h = 1_A$ for some measurable set $A \subset \mathbb{R}$, i.\,e. we estimate the probability of $\{h(T_n) \in A\}$. Here, we clearly have $|\mathfrak{d}_{ij} h| \le c_1$ and $|\mathfrak{d}_i(h-h_0-h_1)| \le c_2$. Consequently, while we cannot achieve the error of Proposition \ref{Bootstrap2} in this situation, we still get an improved consistency result especially for small $m$.

\subsection{Differentiable Functions}

We may apply Theorem \ref{kontinuierlich} in the context of random matrix theory. Here we consider two cases.

\textbf{Case 1} (Wigner matrices). Let $\{\xi_{jk}, 1 \le j \le k \le N \}$ be a family of independent real-valued random variables whose
distributions all satisfy a logarithmic Sobolev inequality \eqref{LSI} with common constant
$\sigma^2$. Putting $\xi_{jk} = \xi_{kj}$ for $1 \le k < j \le N$, we consider a symmetric $N \times N$ random matrix $\Xi =
(\xi_{jk}/\sqrt{N})_{1 \le j, k \le N}$. Denote by $\mu^{(N)} = \mu$ the joint distribution of its ordered eigenvalues $\lambda_1 \le
\ldots \le \lambda_N$ on $\mathbb{R}^N$ (in fact, $\lambda_1 < \ldots < \lambda_N$ a.s.). By a simple
argument using the Hoffman--Wielandt theorem, $\mu$ satisfies a logarithmic Sobolev inequality with constant $\sigma_N^2 =
2 \sigma^2/N$ (see for instance \cite{B-G3}). Note that similar observations also hold for
Hermitian random matrices.

\textbf{Case 2} ($\beta$-ensembles). For $\beta > 0$ fixed, let $\mu_{\beta, V}^{(N)} = \mu^{(N)} = \mu$ be the
probability distribution on $\mathbb{R}^N$ with density given by
\begin{equation}
\label{beta-ensemble}
\mu (d \lambda) = \frac{1}{Z_N} e^{-\beta N \mathcal{H}(\lambda)} d\lambda,\qquad
\mathcal{H}(\lambda) = \frac{1}{2} \sum_{k=1}^{N} V(\lambda_k) - \frac{1}{N} \sum_{1 \le k < l \le N} \log(\lambda_l - \lambda_k)
\end{equation}
for $\lambda = (\lambda_1, \ldots, \lambda_N)$ such that $\lambda_1 < \ldots < \lambda_N$.
Here, $V \colon \mathbb{R} \to \mathbb{R}$ is a strictly convex $\mathcal{C}^2$-smooth function and $Z_N$ is a normalization
constant. It is well-known that for $\beta = 1, 2, 4$, these probability measures correspond to the distributions of the classical
invariant random matrix ensembles (orthogonal, unitary and symplectic, respectively).
For other $\beta$, one can interpret \eqref{beta-ensemble} as particle systems on the real line with Coulomb interactions.
Using the convexity of $V$, we may easily verify that
\begin{equation}
\label{conv}
\mathcal{H}''(\lambda) \ge a \text{Id}
\end{equation}
uniformly in $\lambda$, where $\mathcal{H}''(\lambda)$ denotes the Hessian of $\mathcal{H}$, $\text{Id}$ denotes the $N \times N$
identity matrix and $a > 0$ is some constant. As a consequence, by the classical Bakry-Emery criterion, $\mu$ satisfies a
logarithmic Sobolev inequality \eqref{LSI} with constant $\sigma_N^2 = 1/(aN)$. For a detailed discussion see S.\,G. Bobkov and M. Ledoux \cite{B-L}.

Now consider the probability space $(\mathbb{R}^N, \mathbb{B}^N, \mu)$, where $\mu$ is either the joint eigenvalue
distribution of $\Xi$ or the distribution defined in \eqref{beta-ensemble}. If $f \colon \mathbb{R} \to \mathbb{R}$ is a
$\mathcal{C}^1$-smooth function, it is well-known that asymptotic normality
\begin{equation}
\label{S_N beta}
S_N = \sum_{j=1}^{N} (f(\lambda_j) - \mu[f(\lambda_j)]) \Rightarrow \mathcal{N}(0, \sigma_f^2)
\end{equation}
holds for the self-normalized linear eigenvalue statistics $S_N$. Here, ``$\Rightarrow$'' denotes weak convergence, $\mu[\cdot]$ means integration with
respect to $\mu$ and $\mathcal{N}(0, \sigma_f^2)$ denotes a normal distribution with mean zero and variance $\sigma_f^2$
depending on $f$.

This result goes back to K. Johansson \cite{J} for the case of $\beta$-ensembles and, for general Wigner matrices, A.\,M. Khorunzhy, B.\,A. Khoruzhenko and L.\,A. Pastur
\cite{K-K-P} as well as Ya. Sinai and A. Soshnikov \cite{S-S}. Such results
have been extensively studied since then. Concentration of measure results have been obtained by A. Guionnet and O. Zeitouni \cite{G-Z},
proving concentration inequalities centered at the mean using techniques by Talagrand and Ledoux discussed in the introduction.
In particular, they proved that $S_N$ has fluctuations of order $\mathcal{O}_P(1)$ if $f'$ is absolutely bounded. Here we can complement these results by a second order concentration bound which only requires $f''$ to be absolutely bounded.

\begin{proposition}
\label{secondorderbeta}
Let $\mu$ be the joint distribution of the ordered eigenvalues of $\Xi$ or the $\beta$-ensemble distribution defined in \eqref{beta-ensemble}. Let $f \colon
\mathbb{R} \to \mathbb{R}$ be a $\mathcal{C}^2$-smooth function with $f'(\lambda_j)
\in L^1(\mu)$ and second derivatives bounded by some constant $\gamma > 0$, and let
$\tilde{S}_N := S_N - \sum_{j=1}^{N} (\lambda_j - \mu[\lambda_j]) \mu [f'(\lambda_j)]$
with $S_N$ as in \eqref{S_N beta}. Then, we have
$$\int \exp(cN^{1/2} |\tilde{S}_N|) d\mu \le 2,$$
where $c = c(\gamma) > 0$ is some constant. If $\mu$ is the eigenvalue distribution
of $\Xi$, $c$ moreover depends on the Sobolev constant $\sigma^2$, and if $\mu$
is the $\beta$-ensemble distribution \eqref{beta-ensemble}, $c$ also depends on the
quantity $a$ from \eqref{conv}.
\end{proposition}

Proposition \ref{secondorderbeta} follows from Theorem \ref{kontinuierlich} and the fact that the Sobolev constant $\sigma_N^2$ is of
order $1/N$. In view of the self-normalized property of $S_N$, the fluctuation result for $\tilde{S}_N$ is of the next order,
although the scaling is of order $\sqrt{N}$ only.

Results of this type are useful in situations where $f'$ is not bounded (i.\,e. \cite{G-Z} cannot be applied), in particular if $f$ grows at most quadratically. In this case, concentration results for $S_N$ may be obtained by considering $\tilde{S}_N$ and the ``linear'' part separately. Controlling the linear part is usually an easy task, while $\tilde{S}_N$ can be handled by Proposition \ref{secondorderbeta}. The idea of splitting eigenvalue statistics into a ``linear'' term and a remainder also appears in the analysis of interacting particle systems. That is, in \eqref{beta-ensemble}, another quadratic ``interaction energy'' term of the form $\frac{1}{N} \sum_{i < j} h(\lambda_j-\lambda_i)$ is added to $\mathcal{H}(\lambda)$, where $h$ is a ``kernel'' function with suitable properties. These particle system have been studied by F. G\"{o}tze and M. Venker \cite{G-V}, including a concentration of measure result similar to our bounds for the recentered interaction energy $h$, removing both the expected value and a linear term (cf. e.\,g. (29) and Remark 4.8 there).

Indeed, second order results may also be used for establishing concentration bounds for quadratic eigenvalue statistics. The idea of studying higher order statistics of eigenvalues can already be found in A. Lytova and L. Pastur \cite{L-P}, where the authors proved a law of large numbers and a central limit theorem for $U$-statistics of eigenvalues. Following \cite{G-V}, an interesting question is whether the self-normalization phenomenon extends to what might be informally called a ``double self-normalization'' (at the level of the fluctuations, in our framework). That is, we shall examine whether quadratic statistics which are ``recentered'' in a suitable way may have fluctuations of a better order than $\mathcal{O}_P(N)$, i.\,e. possibly even $\mathcal{O}_P(1)$.

To this aim, we consider a
sufficiently smooth ``kernel'' function $g \colon \mathbb{R}^2 \to \mathbb{R}$ and
set
$T_N := \sum_{j \ne k} g(\lambda_j, \lambda_k).$
Rescaling $T_N - \mu[T_N]$, we arrive at asymptotic normality. To obtain second order
concentration of measure results, we shall also center around a linear correction
term according to \eqref{corrected}. That is, we define $Q_N = Q_N(\lambda)$ by
\begin{align}
\label{Q_N}
\begin{split}
Q_N :=  &\sum_{j \ne k} g(\lambda_j, \lambda_k) - \sum_{j \ne k}
\mu[g(\lambda_j, \lambda_k)]\\
&- \sum_{i=1}^{N}\big(\sum_{k: k \ne i} (\mu[g_x(\lambda_i,
\lambda_k)]+\mu[g_y(\lambda_k,\lambda_i)])\big)(\lambda_i - \mu[\lambda_i]).
\end{split}
\end{align}
Here, $g_x$, $g_y$ etc. denote partial derivatives. For instance, if $g(x,y) = xy$, $Q_N$ has the form
$$Q_N(\lambda) = \sum_{j \ne k} \big((\lambda_j - \mu[\lambda_j])(\lambda_k -
\mu[\lambda_k]) - \mu[(\lambda_j - \mu[\lambda_j])(\lambda_k -
\mu[\lambda_k])]\big).$$
In particular, this demonstrates that it is natural to remove a ``linear'' term in this context (also recall the discussion at the beginning of Section 1).

\begin{proposition}
\label{quadraticstat1}
Let $\mu$ be the joint distribution of the ordered eigenvalues of $\Xi$ or the
$\beta$-ensemble distribution defined in \eqref{beta-ensemble}. Let $g \colon
\mathbb{R}^2 \to \mathbb{R}$ be a $\mathcal{C}^2$-smooth function with first order
derivatives in $L^1(\mu)$ and second order derivatives bounded by some constant $\gamma > 0$. Consider $Q_N$ as defined in \eqref{Q_N}. Then,
for some constant $c = c(\gamma) > 0$,
\begin{align}\label{Aim}
\int \exp\left(\frac{c}{N^{1/2}}\lvert Q_N\rvert\right) d\mu \le 2.
\end{align}
In the special case of $g(x,y) := xy$, we have
\begin{align}\label{Aim2}
\int \exp\left(c\lvert Q_N\rvert\right) d\mu \le 2.
\end{align}
If $\mu$ is the eigenvalue distribution of $\Xi$, $c$ moreover depends on the
Sobolev constant $\sigma^2$, and if $\mu$ is the $\beta$-ensemble distribution
\eqref{beta-ensemble}, $c$ also depends on the quantity $a$ from \eqref{conv}.
\end{proposition}

\begin{proof}
To check the conditions of Theorem \ref{kontinuierlich}, note that the Hessian of $Q_N$ has entries
\begin{align*}
(Q_N'')_{ij} &= g_{xy}(\lambda_i, \lambda_j) + g_{yx}(\lambda_j,
\lambda_i),\qquad i \ne j,\\
(Q_N'')_{ii} &= \sum_{k: k \ne i}(g_{xx}(\lambda_i, \lambda_k) +
g_{yy}(\lambda_k, \lambda_i)).
\end{align*}
Using the boundedness of the second derivatives of $g$ it follows easily that
$\lVert Q_N'' \rVert_\mathrm{Op} \le cN$.
Here, $c = c(\gamma)$ denotes a numerical constant which will vary from line to line throughout the proof.

On the other hand, we have
\begin{align*}
\int \lVert Q_N'' \rVert_\mathrm{HS}^2 d\mu = &\sum_{i \ne j} \int
(g_{xy}(\lambda_i, \lambda_j) + g_{yx}(\lambda_j, \lambda_i))^2 d\mu\\
&+ \sum_{i=1}^{N} \int \Big(\sum_{k: k \ne i}(g_{xx}(\lambda_i, \lambda_k) +
g_{yy}(\lambda_k, \lambda_i))\Big)^2 d\mu.
\end{align*}
Here, the sum corresponding to the off-diagonal terms can clearly be bounded by $cN^2$, while the sum corresponding to the diagonal terms can only be bounded by $cN^3$ in general. Therefore,
$\int \lVert Q_N'' \rVert_\mathrm{HS}^2 d\mu \le cN^3$.

Finally, if $g(x,y) := xy$, we have $g_{xx} \equiv g_{yy} \equiv 0$, and consequently
$\int \lVert Q_N'' \rVert_\mathrm{HS}^2 d\mu \le cN^2$.
Applying Theorem \ref{kontinuierlich} finishes the proof.
\end{proof}

In case of $g(x,y) := xy$, by \eqref{Aim2}, $Q_N$ has fluctuations of order $\mathcal{O}_P(1)$, which can be regarded as an extension of the self-normalizing property to a second order situation at least on the level of the fluctuations of $Q_N$.

Unfortunately, this property does not seem to hold in full strength for general kernels $g$. To explain this, note that $Q_N$ may be decomposed into a ``pure'' quadratic part $D_N$ and a remainder term $Q_N - D_N$, where $D_N = D_N(\lambda)$ is given by
\begin{align}\label{D_N}
\begin{split}
D_N &= \frac{1}{2} \sum_{i=1}^{N} \mu[\partial_{ii} Q_N](\lambda_i^2 - 2\mu[\lambda_i]
\lambda_i + 2\mu[\lambda_i]^2 - \mu[\lambda_i^2]).\\
%&= \frac{1}{2} \sum_{i=1}^{N} \Big(\sum_{k: k \ne i}(\mu[g_{xx}(\lambda_i,
%\lambda_k)] + \mu[g_{yy}(\lambda_k, \lambda_i)]) \Big)(\lambda_i^2 - 2\mu[\lambda_i]
%\lambda_i + 2\mu[\lambda_i]^2 - \mu[\lambda_i^2]).
\end{split}
\end{align}
Arguing similarly as in the proof of Proposition \ref{quadraticstat1}, for an arbitrary kernel $g$, $Q_N - D_N$ has fluctuations of order $\mathcal{O}_P(1)$, while the fluctuations of $D_N$ are of a larger order $\mathcal{O}_P(N^{1/2})$. Therefore, we obtain a factor of $1/N^{1/2}$ in \eqref{Aim}. If $g(x,y) = xy$, we have $D_N \equiv 0$, and hence we arrive at a
different result \eqref{Aim2}.
On the other hand, considering the case of $g(x,y) = f(x)$ for some function $f$ with bounded second order derivatives, we arrive at Proposition \ref{secondorderbeta} again. In particular, this shows that in general, the additional factor $1/N^{1/2}$ in \eqref{Aim} cannot be removed.

\end{document}